\documentclass[letter,10pt]{amsart}

\usepackage[latin1]{inputenc}   
\usepackage{amssymb,amsmath,amsthm,mathrsfs}
\usepackage{graphicx}
\usepackage[all]{xy}
\usepackage[usenames,dvipsnames]{color}
\usepackage{enumitem}

\usepackage[colorlinks=true,pagebackref]{hyperref}

 \usepackage{tikz}
\usetikzlibrary{arrows,decorations.pathmorphing,decorations.pathreplacing,positioning,shapes.geometric,shapes.misc,decorations.markings,decorations.fractals,calc,patterns,matrix}

\tikzset{>=stealth',
     cvertex/.style={circle,draw=black,inner sep=1pt,outer sep=3pt},
     vertex/.style={circle,fill=black,inner sep=1pt,outer sep=3pt},
     star/.style={circle,fill=yellow,inner sep=0.75pt,outer sep=0.75pt},
     tvertex/.style={inner sep=1pt,font=\criptsize},
     gap/.style={inner sep=0.5pt,fill=white}}

\newcommand{\Z}{\ensuremath{\mathbb{Z}}}
\newcommand{\C}{\ensuremath{\mathbb{C}}}

\newcommand{\N}{\ensuremath{\mathbb{N}}}

\newcommand{\ra}{\ensuremath{\rightarrow}}
\newcommand{\lra}{\ensuremath{\longrightarrow}}

\newcommand{\tr}{\operatorname{\mathfrak{T}}\nolimits}

\DeclareMathOperator{\Spec}{Spec}

\DeclareMathOperator{\Hom}{Hom}
\DeclareMathOperator{\Ext}{Ext}
\DeclareMathOperator{\End}{End}
\DeclareMathOperator{\add}{add}

\newcommand{\mc}[1]{\ensuremath{\mathcal{#1}}}   
\newcommand{\mf}[1]{\ensuremath{\mathfrak{#1}}}   

\newcommand{\gl}{\operatorname{gl.dim}\nolimits}

\newcommand{\gs}{\operatorname{gs}\nolimits}
\newcommand{\img}{\operatorname{Im}\nolimits}
\newcommand{\mmod}[1]{\operatorname{mod}(#1)}

\newcommand{\GL}{\operatorname{GL}}
\newcommand{\CM}{\operatorname{{MCM}}}

\newcommand{\pd}{\operatorname{proj.dim}\nolimits}

\theoremstyle{theorem}
\newtheorem{Thm}{Theorem}[section]         
\newtheorem{Lem}[Thm]{Lemma}
\newtheorem{Cor}[Thm]{Corollary}
   
\newtheorem{Prop}[Thm]{Proposition}   
\newtheorem{Conj}[Thm]{Conjecture}
\newtheorem{Qu}[Thm]{Question}

\theoremstyle{remark}
\newtheorem{Bem}[Thm]{Remark}
\newtheorem{ex}[Thm]{Example}
\newtheorem{Construction}[Thm]{Construction}

\theoremstyle{definition}
\newtheorem{defi}[Thm]{Definition} 

\title{Computing global dimension of endomorphism rings via ladders}

\author{Brandon Doherty}
\address{Department of Mathematics and Statistics, University of New Brunswick, Fredericton, NB. E3B 5A3, Canada}
\email{Brandon.Doherty@unb.ca}

\author{Eleonore Faber}
\address{Fakult\"at f\"ur Mathematik,
Universit\"at Wien, 
Oskar-Morgenstern-Platz 1,
A-1090 Wien,
Austria}
\curraddr{Department of Mathematics, University of Michigan, Ann Arbor, MI 48109, USA}
\email{eleonore.faber@univie.ac.at}

\author{Colin Ingalls}
\address{Department of Mathematics and Statistics, University of New Brunswick, Fredericton, NB. E3B 5A3, Canada}
\email{cingalls@unb.ca}

\date{\today}
\thanks{
B.D.~was partially supported by an NSERC Undergraduate Summer Research Award.
E.F. was partially supported by the Austrian Science Fund (FWF) in frame of project J3326 and gratefully acknowledges support by the Institut Mittag-Leffler (Djursholm, Sweden). 
C.I.~was partially supported by an NSERC Discovery Grant.
} 

\begin{document}

\begin{abstract}
This paper deals with computing the global dimension of endomorphism rings of maximal Cohen--Macaulay (=$\CM$) modules over commutative rings.   Several examples are computed. In particular, we determine the global spectra, that is, the sets of all possible finite global dimensions of endomorphism rings of $\CM$-modules, of the  curve singularities of type $A_n$ for all $n$, $D_n$ for $n \leq  13$ and $E_{6,7,8}$ and compute the global dimensions of Leuschke's normalization chains for all ADE curves, as announced in \cite{DFI}. Moreover, we determine the centre of an endomorphism ring of a $\CM$-module over any curve singularity of finite $\CM$-type. \\ 
In general, we describe a method for the computation of the global dimension of an endomorphism ring $\End_R M$, where $R$ is a Henselian local ring, using $\add(M)$-approximations. When $M\neq 0$ is a $\CM$-module over $R$ and $R$ is Henselian local of Krull dimension  $\leq 2$ with a canonical module and of finite $\CM$-type, we use Auslander--Reiten theory and Iyama's ladder method to explicitly construct these approximations. 
\end{abstract}

\maketitle

{\footnotesize\tableofcontents}

\section{Introduction}

Let $R$ be a commutative noetherian ring,  $M\neq 0$ be a $\CM$-module over $R$ and set $A=\End_R M$. This article is concerned with the problem of computing the global dimension of $A$, that is, the smallest number $n$ such that any $A$-module has a projective resolution of length $\leq n$. In most cases, $n$ will be infinity: for example, if $A$ itself is commutative, then $\gl A < \infty$ if and only if $A$ is a regular ring, by Serre's well-known theorem. But if $A$ is noncommutative, the situation is much more involved and  one is precisely interested in this case: in recent years the study of endomorphism rings of finite global dimension has become increasingly popular, since they appear as analogues of resolutions of singularities of $\Spec(R)$, namely as so-called noncommutative (crepant) resolutions of singularities (=NC(C)Rs). In the original treatment of van den Bergh \cite{vandenbergh04} NCCRs over Gorenstein normal domains $R$ were defined as homologically homogeneous endomorphism rings of finitely generated reflexive $R$-modules, which implies that their global dimension is automatically equal to the Krull dimension of $R$. NCCRs were further studied in, e.g., \cite{BLvdB10, IyamaWemyss10} and see \cite{Leuschke12} for an overview. In \cite{DaoIyamaTakahashiVial} and \cite{DFI}, more general NCRs of any commutative ring $R$ were defined as endomorphism rings of finitely generated modules of full support and of finite global dimension. In \cite{DFI} the \emph{global spectrum} of $R$ was introduced as the set of all possible finite global dimensions of endomorphism rings of $\CM$-modules\footnote{We only consider $\CM$-modules because there is a rich structure theory which also features representation theoretic methods, having their origin in representation theory of Artin algebras, see \cite{LeuschkeWiegand, Yoshino90}.}. The global spectrum is a quite mysterious object and so far there have only been very few explicit examples computed. Thus, obtaining more examples of global spectra was one of the main motivations for the present work. \\
 It is supposed that the global spectrum is strongly related to categorical invariants, like the Orlov spectrum, which appears in birational geometry \cite{BallardFaveroKatzarkov}. Moreover, knowing more about possible values of global dimensions of endomorphism rings also can be helpful in understanding more classical topics in commutative algebra, in particular, the Grauert--Remmert normalization algorithm for curves, see \cite{Leuschke07} and \cite{IyamaRejective}, where endomorphism rings appearing in this algorithm are studied.  
 \\

In this paper we present a method to compute $\gl A$, which uses combinatorics on a particular quiver related to $R$, the Auslander--Reiten (=AR) quiver.
We will use \emph{ladders}, which have originally been introduced by Igusa and Todorov in order to prove the radical layers theorem for Artin algebras \cite{IgusaTodorovRadical}. In 2005, Iyama generalized their methods to orders
\cite{IyamaTau1,IyamaTau2, IyamaTau3}.  
More recently Iyama and  Wemyss used ladders to compute syzygies of $\CM$-modules over two-dimensional quotient singularities, see \cite{IyamaWemyss10a} and also Wemyss in \cite{WemyssGL}. In these papers, the authors used ladders to compute $\add(R)$-approximations in order to compute syzygies. On the other hand, Quarles used ladders implicitly in order to compute NCRs of some two-dimensional quotient singularities in his thesis \cite{Quarles}. \\
The present paper grew out of trying to understand their methods and applying them in order to  explicitly compute $\add(M)$-approximations of $\CM$-modules and from these global dimensions of endomorphism rings. We follow the treatment of Iyama and Wemyss closely but consider ladders in a slightly different context, in particular for $1$-dimensional Cohen--Macaulay-rings. Moreover, we try to formulate the theorems in the most general context, so instead of complete local rings, we consider Henselian local rings. We also study the boundaries of generalizing this method. In particular, requirements on the dimension of $R$ and finite $\CM$-type.  \\


Let us comment on the specific results: 
We give an algorithm to compute the global dimension of $\End_R M$. The  input data for the algorithm are $M$ and the AR-quiver of $R$. The algorithm enables us to compute the global spectrum of $R$, namely the set of all possible finite global dimensions of endomorphism rings of $\CM(R)$-modules. 
Here we will compute the global spectrum of the $A_n$ for all $n$, $E_{6,7,8}$ and $D_n$ for $n \leq 13$, curve singularities (Thm.~\ref{Thm:globalspecsimple})  and give conjectural global spectra for the remaining $D_n$ curve singularities. The evidence for these conjectures are based on computations of a SAGE program written by B.~Doherty. We also prove the claimed global dimension of  Leuschke's \cite{Leuschke07} endomorphism rings of  \cite[Section 3.1.1]{DFI} in Section \ref{Sub:Leuschke}. Moreover, we  give a few examples of computations of global spectra of some two dimensional singularities, using the SAGE program.\\

In noncommutative algebra, the centre $Z(A)$ of a noncommutative algebra $A$, is a fundamental invariant. For endomorphism rings $A=\End_R M$ over a reduced noetherian commutative ring and $M$ a faithful torsion free module, we show that $Z(A)$ is the largest finite extension $R \subseteq R' \subseteq Q(R)$ such that $M$ is still defined over $R'$, i.e., $R'M =M$, where multiplication is induced by the multiplication of  $Q(R)$ on $Q(R)M$ (see Theorem \ref{Thm:centres}).  In the case of endomorphism rings over curves $R$ of finite $\CM$-type we are able (by a case by case analysis) to determine the centres explicitly. This also shows how the AR-quivers of overrings of $R$ sit in the AR-quiver of $R$. Knowing the centres of the endormorphism ring allows a significant decrease in running time of the computation of global dimension. \\

An outline of the paper is as follows:  in Section \ref{Sec:general} and \ref{Sec:tauladders} the necessary notions and techniques for computing the global dimension of endomorphism rings using ladders are recalled. We show that they work in our Henselian setting and give complete statements rather than to refering to the literature. In particular, we revisit the Iyama--Wemyss theorem about the existence of ladders in our Henselian setting, see Theorem \ref{Thm:ladder}.  \\
The main examples are computed in Section \ref{Sec:examples}: in \ref{Sub:Leuschke} we prove the claimed global dimensions for the ADE curve singularities from \cite{DFI}. In Theorem \ref{Thm:specAn} we prove using ladders that the global spectrum of $A_n$ curves is $\{1,2\}$ for even $n$ and $\{1,2,3\}$ for odd $n$. In \ref{Sub:globalspecsimple}  the global spectra of ADE curves of types $A_n$, $D_n$ ($n \leq 13$), $E_{6,7,8}$ are computed with the help of a SAGE program, which is explained there. We also compute some global spectra of surfaces of finite $\CM$ type. Then we discuss examples and the case of infinite $\CM$-type in Section \ref{Sub:infiniteCM}. In particular, we show in Prop.~\ref{Prop:infiniteGorenstein} that for a normal $2$-dimensional Gorenstein Henselian local ring, ladders will not yield an $\add(M)$-approximation in general. \\
In Section \ref{Sec:centre} we determine centres of endomorphism rings of torsion free modules over reduced noetherian rings, Theorem \ref{Thm:centres}. In the following the centres of endomorphism rings of all $\CM$ modules over one dimensional local rings of finite $\CM$-type are explicitly determined.

\section{Setting and generalities} \label{Sec:general}

In this section we recall several notions and discuss how to construct minimal projective resolutions of simple modules of endomorphism rings of $\CM$-modules over certain rings $R$, using approximations of modules over $R$. \\

In this paper $R$ will always denote a commutative noetherian ring. We will assume that $R$ is either artinian or local Henselian of positive Krull dimension. In the latter case we will try to state the theorems for $R$ as general as possible, but all theorems work if one assumes $R$ complete local (plus additional hypotheses on dimension etc.).   Any full subcategory of $\CM(R)$ should be closed under direct summands and finite direct sums and will be called {\it additive}. \\

We denote by $\mmod{R}$  the category of finitely generated modules and by $\CM(R)$ the full subcategory of maximal Cohen--Macaulay modules. Let $M \in \mmod{R}$ and denote by $\add(M)$ the \emph{additive closure of $M$}, that is, the full subcategory of $\mmod{R}$ consisting of direct summands of finite direct sums of $M$. \\
Let $\mc{C}$ be a full subcategory of $\mmod{R}$, then morphisms between two objects $X, Y \in \mc{C}$ are denoted by $\Hom_\mc{C}(X,Y)= \mc{C}(X,Y)$. 
Note that, since $\mc{C}$ is full, $\mc{C}(X,Y)=\Hom_R(X,Y)$. 
The Jacobson radical of a (not necessarily commutative) ring $\Lambda$, i.e., the intersection of all right maximal ideals of $\Lambda$, 
is denoted by ${\bf J}_\Lambda$. We will also need some more general notions concerning categories: Let $\mc{C}$ be an additive category, then an \emph{ideal $\mc{I}$ in \mc{C}} is a selection of a subgroup $\mc{I}(X,Y) \subseteq \Hom_\mc{C}(X,Y)$ for any $X,Y \in \mathrm{Ob}(\mc{C})$ such that for any $h \in \mc{C}(X,Y)$, $f \in \mc{I}(Y,Z)$ and $g \in \mc{C}(Z,W)$ one has $gfh \in \mc{I}(X,W)$ (cf.~\cite{KellyRadical}). We define the Jacobson radical, an ideal of $\mc{C}$  by ${\bf J}_\mc{C}(X,Y)=\{ f \in \Hom_{\mc{C}}(X,Y): hfg$ is not an isomorphism for any $g: A \ra X$ and $h: Y \ra A$ for any indecomposable $A \in \mc{C} \}$. The \emph{higher radicals} ${\bf J}^n_{\mc{C}}(X,Y)$ for $n \geq 2$ are defined to be $\{f \in \Hom_{\mc{C}}(X,Y)$: there exist an $A \in \mc{C}$ and morphisms $g \in {\bf J}_{\mc{C}}(X,A)$ and $h \in {\bf J}_{\mc{C}}^{n-1}(A,Y)$ with $f=hg \}$. For a full subcategory $\mc{D}$ of $\mc{C}$ one denotes by $[\mc{D}]$ the ideal of $\mc{C}$ of all morphisms in $\mc{C}$ which factor through an object in $\mc{D}$. The \emph{factor category} $\mc{C}/[\mc{D}]$ consists of the same objects as $\mc{C}$ and the morphisms between two objects $X,Y$ are $\Hom_R(X,Y)$ modulo those factoring through an object of $\mc{D}$. An object $X$ in $\mc{C}$ is called \emph{(strongly) indecomposable} if $\End_\mc{C} X$ is local.  A category $\mc{C}$ is called \emph{Krull-Schmidt} if any object is isomorphic to a finite direct sum of indecomposable objects. \\
Let $\Lambda$ be a noetherian semi-perfect ring. Then any finitely generated $\Lambda$-module $X$ has a \emph{minimal projective resolution}:
$$ \cdots \xrightarrow{} P_2 \xrightarrow{f_2} P_1 \xrightarrow{f_1} P_0 \xrightarrow{f_0} X \lra 0,$$
that is, any $P_i \xrightarrow{f_i} \ker(f_{i-1}) \lra 0$ and $P_0 \lra X \lra 0$ are projective covers.  

\begin{Bem}
In order to compute global dimension of endomorphism rings, we will compute minimal projective resolutions over $\End_RM$. Thus,  we assume $R$ to be artinian or Henselian local for $\End_R M$ to be semi-perfect, see \cite[section 6]{Reiner75}. For such rings, $\mmod{R}$ (and hence $\CM(R)$) is a Krull--Schmidt category.  In Section \ref{Sec:tauladders}  we will consider generalized AR-theory over $R$, in the sense of \cite{IyamaTau1} and see which additional assumptions on $R$ are necessary to apply this theory.
\end{Bem}

The next lemma shows that it is enough to consider {\it basic} endomorphism rings of $\CM$-modules, that is, endomorphism rings of modules where every direct summand appears with multiplicity $1$. 
\begin{Lem} \label{Lem:MoritaReduced}
Let $A=\End_R(L)$, where $L=\bigoplus_{i=1}^n L_i^{a_i}$, where $L_i \in \CM(R)$ are indecomposables and $a_i \geq 1$. Then $A$ is Morita-equivalent to $A_{red}=\End_R(\bigoplus_{i=1}^n L_i)$. In particular $\gl A= \gl A_{red}$.
\end{Lem}

The proof is straightforward, e.g., by applying the criterion for Morita equivalence \cite[Prop. 3.5.6]{McConnellRobson}.


\begin{Lem} \label{Lem:radicaldecomposition}
Let $M, N$ be contained in $\mmod{R}$. Then ${\bf J}_{\mmod{R}}(M,N)=\Hom_R(M,N)$ if and only if $M$ and $N$ do not have any direct summand in common. Moreover, ${\bf J}(\bigoplus_i X_i, \bigoplus_j Y_j)=\bigoplus_{ij} {\bf J}(X_i,Y_j)$ for indecomposable $X_i$, $Y_i$. This decomposition also holds for the higher radicals, i.e.,  ${\bf J}^k(\bigoplus_i X_i, \bigoplus_j Y_j)=\bigoplus_{ij} {\bf J}^k(X_i,Y_j)$ for $k \geq 0$. 
\end{Lem}

\begin{proof} The first two assertions are straightforward computations, see e.g., \cite{Quarles} Prop. 2.1.4, (b) and (f).
The last assertion is easy to verify, using that $f \in {\bf J}^2(\bigoplus_i X_i, \bigoplus_j Y_j)$ if and only if $f$ is a composition of radical morphisms, and induction.
\end{proof}

\subsection{Global dimension and projective modules} \label{sub:structureEnd}

Let $\Lambda$ be a 
noetherian ring with $1$. The (left or right) \emph{global dimension} of a ring $\Lambda$ with $1$ is defined to be 
$$\gl (\Lambda)=\sup\{ \pd(M): M \in \mathrm{Mod}(\Lambda) \}.$$
As explained in more detail e.g.~in Section $5$ of \cite{DFI}, if $\Lambda$ is a finitely generated algebra over a noetherian ring, then $$\gl \Lambda= \max \{ \pd S: S \text{ a simple }\Lambda \text{-module} \}.$$

If  $\mmod{\Lambda}$ is Krull-Schmidt, then $\Lambda$ is semi-perfect, i.e., every finitely generated $\Lambda$-module has a projective cover.
Supposing $\Lambda=\End_R M$, over $R$ as above, one knows from the structure theorem of projective modules for $\Lambda$ that each indecomposable projective $P_i$ is of the form $\Hom_R(M,M_i)$, where $M=\bigoplus_i M_i$. For each projective there is a simple $S_i=P_i/{\bf J}_\Lambda \cap P_i$,   and these are all the simples.  In particular, one has  an equivalence of categories: 
$$\add(M) \simeq \mc{P}(\Lambda),$$
where $\mc{P}(\Lambda)$ denotes the fully faithful subcategory of projective modules of $\mmod{\Lambda}$, via the projectivization functor $\Hom_R(M,-)$ (see e.g. \cite{Auslander-Reiten-Smalo}).

\begin{Bem}
In the situation $\Lambda = \End_R M$ with $R$ a local Henselian ring and $M \in \mmod{R}$, with the additional assumption that the residue field is separably closed, one can use the method for computation of the simples as outlined in \cite[Section 5.3.1]{DFI}, following the exposition of \cite{Hazewinkeletc}. However, in order to use this method, one not only has to know the indecomposable summands $M_i$ of $M$ but should moreover have an explicit knowledge of $\Hom_R(M_i,M_j)$ and ${\bf J}(\End_RM_i)$ for any $i,j$. 
 \end{Bem}

 \subsection{Right approximations and right almost split maps}

Right approximations and right almost split maps by full subcategories of $\CM(R)$ will provide the framework for constructing projective resolutions over endomorphism rings. In other words, 
we reduce the problem of constructing projective resolutions of modules over $\End_R(M)$ to that of constructing approximations by the subcategory $\add(M)$ over $R$.  The approximations will be constructed in Section \ref{Sec:tauladders}, under some conditions on $R$, using almost split maps over $R$ and the ladder technique.

\begin{defi} 
Let $\mc{C}$ be a full subcategory of $\CM(R)$. Let  $0 \ra Z \ra Y \stackrel{f}{\ra} X$ be an exact sequence in $\CM(R)$. We say that $f$ is a \emph{right $\mc{C}$-approximation} if $Y \in \mc{C}$ and all $g: Y' \rightarrow X  \in \mc{C}$ factor through $f$, that is,
$$ \mc{C}(-,Y) \xrightarrow{\Hom_R(f)} \mc{C}(-,X) \lra 0.$$
In other words: for any $g: Y' \lra X$ with $Y' \in \mc{C}$ there exists an $h: Y' \lra Y$ such that $fh=g$.
We say that $f$ is a \emph{right almost split map in $\mc{C}$} (or \emph{right \mc{C}-almost split map})  if $Y \in \mc{C}$, $f \in {\bf J}_{\mc{C}}(Y, X)$ and for any $Y' \in \mc{C}$ and $g \in {\bf J}_{\mc{C}}(Y',X)$ there exists a $h: Y' \lra Y$ such that $fh=g$, that is, the sequence
$$ {\mc{C}}(-, Y) \xrightarrow{\Hom_R(f)} {\bf J}_{\mc{C}}(-,X) \lra 0$$
is exact. \\
A right $\mc{C}$-approximation (or right $\mc{C}$-almost split map) $f: Y \lra X$ is \emph{(right) minimal} if every morphism $\alpha: Y \lra Y$ such that $f \alpha= f$ is an isomorphism. It can be easily seen 
that this is equivalent to saying that $f$ is right minimal if there does not exist a non-zero direct summand of $Y$ that is mapped to zero under $f$. 
 \\
An exact sequence $\xi: \cdots \xrightarrow{f_2} Y_1 \xrightarrow{f_1} Y_0 \xrightarrow{f_0} X$ is called a \emph{$\mc{C}$-resolution of $X$} if each $Y_i \in \mc{C}$ and $\mc{C}(-,\xi): \cdots \xrightarrow{\Hom(f_2)} \mc{C}(-,Y_1) \xrightarrow{\Hom_R(f_1)} \mc{C}(-,Y_0) \xrightarrow{\Hom(f_0)} \mc{C}(-,X) \xrightarrow{} 0$ is exact on $\mc{C}$. 
A $\mc{C}$-resolution is \emph{minimal} if each $f_i$ is right minimal. 
\end{defi}

By definition the direct sum of right $\mc{C}$-split morphisms (\mc{C}-approximations) of some $X, X'$ is again a right \mc{C}-almost split morphism ($\mc{C}$-approximation) of the direct sum $X \oplus X'$. \\
Dually, one can also define \emph{left $\mc{C}$-approximations}, \emph{left $\mc{C}$-almost split maps} and \emph{left $\mc{C}$-resolutions}. \\

\begin{ex} (1) If $Z$ is in $\mc{C}$, then $0 \ra Z \xrightarrow{Id_Z} Z \ra 0$ is the right minimal $\mc{C}$-approximation. \\
(2) Let $R$ be a local Henselian CM ring with canonical module that is an isolated singularity. Then any indecomposable $M \in \CM(R)$, which does not contain $R$ as a factor, has a unique minimal right $\CM(R)$-almost split map, coming from its AR-sequence, see e.g. \cite[Chapter 13]{LeuschkeWiegand}. 
\end{ex}

A full subcategory $\mc{C} \subseteq \CM(R)$ is called \emph{contravariantly (covariantly) finite} in $\CM(R)$ if there exist right (left) $\mc{C}$-approximations for all objects in $\CM(R)$. Moreover, $\mc{C}$ is called \emph{functorially finite} if $\mc{C}$ is both co- and contravariantly finite. By Auslander--Buchweitz \cite{AuslanderBuchweitz} for any functorially finite subcategory $\mc{C}$ of $\CM(R)$, any $X \in \mmod{R}$ has a minimal right (left) $\mc{C}$-resolution. \\
In general it is not clear  whether a subcategory $\mc{C} \subseteq \CM(R)$ is contravariantly finite, for more properties and examples see e.g., \cite{AuslanderSmalo} for Artin algebras, \cite{IyamaHigher} for maximally orthogonal subcategories (=cluster tilting subcategories) and \cite{TakahashiResolving} for resolving subcategories. \\
In our case we have the following: 

\begin{Lem} \label{Lem:addMcontra}
Let $\mc{C}=\add(L)$ for some $\L \in \CM(R)$. Then $\mc{C}$ is contravariantly finite.
\end{Lem}

\begin{proof}
The proof is the same as in the artinian case, see \cite[Prop.4.2]{AuslanderSmalo}. The facts needed are that $\Hom_R(L,N)$ of two noetherian finitely generated modules is finitely generated as $\End_R L$-module and that $\mc{C}$ has only finitely many indecomposables.
\end{proof}

The following lemma gives another useful characterization of minimal right $\mc{C}$-almost split maps and -approximations:

\begin{Lem} \label{Lem:minimalres}
Let $\mc{C}$ be a full subcategory of $\CM(R)$ and let $f: Y \lra X$ with $Y \in \mc{C}$ be either a right $\mc{C}$-approximation or $\mc{C}$ almost split map. \\
(1) Then $f$ is right-minimal if and only if in 
in exact sequence
$$0 \lra \ker(f) \xrightarrow{g} Y \xrightarrow{f} X$$
the morphism $g \in {\bf J}(\ker(f), Y)$. Moreover if a $\mc{C}$-resolution is minimal, then each $f_i \in {\bf J}_{\mc{C}}$ for $i \geq 1$. \\ 
(2) If $X$ has a right $\mc{C}$-approximation (-almost split map) $f: Y \lra X$, then there exists a minimal right $\mc{C}$-approximation (-almost split map). Moreover, if $\mc{C}$ is contravariantly finite, any $X$ has a minimal $\mc{C}$-resolution.
\end{Lem}

\begin{proof}
(1) Let $Z:=\ker(f)$. First suppose that $f$ is minimal. Suppose that there  exists a $g$ not in ${\bf J}(Z,Y)$. Since $\mmod{R}$ is Krull--Schmidt, we may decompose $Z=\bigoplus_{i=1}^m Z_i$ and $Y=\bigoplus_{j=1}^n Y_j$ into indecomposables. Then $g$ can be decomposed into $g_{ij}: Z_i \lra Y_j$, and  $\Hom_R(Z,Y)$ (resp.~the radical ${\bf J}(Z,Y)$) decomposes into direct summands $\Hom_R(Z_i,Y_j)$ (resp.~${\bf J}(Z_i, Y_j)$ by Lemma \ref{Lem:radicaldecomposition}). So if $g$ is not in the radical, then at least one $g_{ij} \not \in {\bf J}(Z_i,Y_j)$. But $g_{ij} \not \in {\bf J}(Z_i,Y_j)$ means that $g_{ij}$ is an isomorphism, since $Y_j, Z_i$ are indecomposable. Thus w.l.o.g. $i=j=1$, so that $Z_1 \cong Y_1$. However, this implies that $Y_1$ lies in the kernel of $f$ and hence $f|_{Y_1}=0$, which contradicts the minimality of $f$.  \\
Now suppose that $\cdots \xrightarrow{f_2} M_1 \xrightarrow{f_1} M_0 \xrightarrow{f_0} X$ is a minimal $\mc{C}$-resolution of $X$. Suppose that some $f_i: M_i \ra M_{i-1}$ is not in ${\bf J}_\mc{C}$. This means that $f_i$ is an isomorphism on at least one indecomposable summand $M_{ij}$ of $M_i$. Thus $M_{ij} \subseteq \img(f_i)=\ker(f_{i-1})$, which implies $f|_{M_{ij}}=0$. By definition $f_{i-1}$ is not minimal. \\
(2) The first assertion simply follows from splitting off the direct summands of $Y$ lying in the kernel of $f$. The second assertion can be seen similarly.
\end{proof}

It follows that  a right minimal $\mc{C}$-almost split map (-approximation) of a direct sum $X_1 \oplus X_2$ is given by the direct sum of the right minimal $\mc{C}$-almost split maps (-approximations) of the factors. 

The next lemma follows immediately from the definition of $\mc{C}$-approximation:
\begin{Lem} \label{Lem:exactapprox}
Let $\mc{C}\subseteq \CM(R)$ be a full subcategory, $M \in \mc{C}$, and let  $Y \stackrel{f}{\ra} Z$ be a right $\mc{C}$-approximation with kernel $X$.  Then 
$$0 \ra \mc{C}(M, X) \ra \mc{C}(M, Y) \ra \mc{C}(M,Z) \ra  0$$
is exact.
\end{Lem}

\begin{Lem} \label{Lem:radicalminimal}
Let  $\mc{C}=\add(M)$ for some $M \in\CM(R)$ and $X \in \mmod{R}$ and suppose that 
$$ \xi: \cdots  \xrightarrow{} M_1 \xrightarrow{f_1} M_0 \xrightarrow{f_0} X$$
 is a $\mc{C}$-minimal resolution, i.e., $M_i \in \mc{C}, f_i$ right minimal and
$$ \mc{C}(M, \xi): \cdots \xrightarrow{} \mc{C}_R(M,M_1) \xrightarrow{\Hom_R(M,f_1)} \mc(M,M_0) \xrightarrow{\Hom_R(M,f_0)} \mc{C}(M,X) \xrightarrow{} 0$$
is exact. Then $\mc{C}(M, \xi)$ is a minimal projective $\End_R(M)$-resolution of $\mc{C}(M,X)=\Hom_R(M,X)$. Thus, in particular, the length of $\mc{C}(M,\xi)$ is equal to $\pd_{\End_RM}\Hom_R(M, X)$.
\end{Lem}

\begin{proof}
Clearly $\mc{C}(M,\xi)$ is a projective resolution of $\Hom_R(M,X)$ over $\End_R M$. We have to show its minimality, i.e., that each $\mc{C}(M,M_i) \xrightarrow{\Hom_R(M,f_i)} \mc{C}(M, \ker f_{i-1})$ is a projective cover. 
This is equivalent to showing that each $\Hom_R(M,f_i)$ is a right minimal epimorphism onto its image, see \cite[Prop.~I.4.1]{Auslander-Reiten-Smalo} 
Suppose therefore that $\mc{C}(M,f_i)$ is not right minimal, that is, there exists a direct summand $\mc{C}(M, M_{i\alpha})$ of $\mc{C}(M,M_i)$ such that $\Hom_R(M,f_i)|_{\mc{C}(M, M_{i\alpha})} \equiv 0$. Without loss of generality assume that $\alpha=1$. Take now the projection $\pi_{i1}: M_i \lra M_{i1}$. Then $\Hom_R(M,f)(\pi_{i1})=0$ in $\mc{C}(M,M_j)$. But this means that for any $(m_{i1}, \ldots, m_{ik}) \in M_i=\bigoplus_{\alpha=1}^k M_{i\alpha}$ one has $\Hom_R(M,f_i)(\pi_{i1})(m_i)=f_i \pi_{i1}(m_i)=0$, so that it follows that $M_{i1}$ lies in the kernel of $f_i$. This contradicts minimality of $f_i$.
\end{proof}

\begin{Lem} \label{Lem:resbegin}
Let $\mc{C}=\add(M)$ for a $M \in \CM(R)$ and let $\xi: 0 \ra X \ra Y \stackrel{f}{\ra} Z$ be a (minimal) right  $\mc{C}$-almost split map (in particular: $Y \in \mc{C}$).  If $Z \in \mc{C}$ is indecomposable, then $\mc{C}(M,\xi)$ yields the beginning of a (minimal) projective resolution of the corresponding simple $S_Z$ in $\mmod{\End_R(M)}$:
$$ 0 \ra \mc{C}(M, X) \ra \mc{C}(M, Y) \xrightarrow{\Hom_{R}(M,f)} \mc{C}(M,Z) \ra S_Z \ra 0.$$
Here $\mc{C}(M,Y)$ and $\mc{C}(M,Z)$ are projective over $\End_RM$ and $S_Z \cong \Hom_R(M,Z) / {\bf J}_R(M,Z)$.
\end{Lem}

\begin{proof}
From the definition of right $\mc{C}$-almost split map it follows that the sequence $0 \ra \Hom_R(M, X) \ra \Hom_R(M, Y)  \ra {\bf J}_R (M,Z) \ra 0$ is exact. Splicing this sequence together with the obvious exact sequence $0 \ra {\bf J}_R(M,Z) \ra \Hom_R(M,Z) \ra S_Z \ra 0$ yields the claim.
\end{proof}

\begin{Lem} \label{Lem:existalmostsplit}
Let $R$ be an isolated singularity with $\dim R \leq 2$. If $\mc{C}$ is contravariantly finite in $\CM(R)$, then each $X \in \CM(R)$ has a right $\mc{C}$-almost split map.
\end{Lem}

\begin{proof}
From the assumptions on $R$ it follows that $\CM(R)$ is a right $\tau$-category (see definition \ref{Def:taucat}). Take the (right) $\tau$-sequence $0 \lra \tau X \lra \vartheta X \xrightarrow{f} X$ of $X$. Since $\mc{C}$ is contravariantly finite, $\vartheta X$ has a $\mc{C}$-approximation $Y \xrightarrow{a}  \vartheta X$ with $Y \in \mc{C}$. Using the definitions of $\mc{C}$-approximation and $\tau$-sequence, it follows that any $\alpha: Y' \lra X$, $Y' \in \mc{C}$ and $\alpha \in {\bf J}_{\mc{C}}(Y',X)$ factors through $Y \xrightarrow{fa} X$. From the depth lemma and the assumptions on $\dim R$ it follows that the kernel of $fa$ is in $\CM(R)$ and thus $fa$ is a right $\mc{C}$-almost split map.
\end{proof}

\begin{Bem}
This lemma shows that for the existence of a right $\mc{C}$-almost split map it is needed that the Krull dimension of $R$ is less than or equal to $2$. For rings of higher Krull dimension Iyama has developed higher AR-theory, see \cite{IyamaHigher}. 
\end{Bem}

\subsection{Construction of projective resolutions of the simples of an endomorphism ring} \label{Sub:computationproj}

Let $R$ be a  commutative local noetherian Henselian ring and let $\mc{C}$ be a full contravariantly finite subcategory in $\CM(R)$. Suppose that any $X \in \mc{C}$ has a right $\mc{C}$-almost split map.
Here we show how to construct a (minimal) $\mc{C}$-resolution out of $\mc{C}$-approximations and how this will yield minimal resolutions of the simple modules of an endomorphism ring of $M \in \mc{C}=\add(M)$. We will formulate the construction for contravariantly finite subcategory $\mc{C}=\add(M)$ (cf. Lemma \ref{Lem:addMcontra})  and resolutions of the simples of $\End_R M$. But one can more generally compute minimal $\mc{C}$-resolutions for $\mc{C}$ with the assumptions above.

\begin{Bem}
The assumptions are satisfied if $R$ as above is an isolated singularity, that is, $R_\mf{p}$ is regular for any non-maximal prime ideal $\mf{p}$ of $R$, and $\dim R \leq 2$ and $\mc{C}$ a contravariantly finite subcategory of $\CM(R)$: then $\CM(R)$ is a right $\tau$-category, see Def.~\ref{Def:taucat} and by Lemma \ref{Lem:existalmostsplit} right $\mc{C}$-almost split sequences exist. \\
If $R$ is additionally of finite $\CM$-type, then Theorem \ref{Thm:ladder} explicitly computes such minimal $\mc{C}$-approximations and minimal $\mc{C}$-almost split maps.
 \end{Bem}

\begin{Construction} \label{Const:projres}
Let $M \in \CM(R)$ and $A=\End_RM$ and set $\mc{C}=\add(M)$, $\mc{D}=\CM(R)/[\mc{C}]$. We may suppose that $A$ is basic. As discussed in \ref{sub:structureEnd} there are only finitely many simple $A$-modules, namely, for each indecomposable $Z \in \mc{C}$ one has $S_Z=\Hom_R(M,Z) / {\bf J}_R(M,Z)$. By Lemma \ref{Lem:resbegin} one obtains the beginning of the projective resolution of $S_Z$ from a right $\mc{C}$-almost split map of $Z$:
by assumption there exists an exact sequence with a minimal right $\mc{C}$-almost split map $f_0$
\begin{equation} \label{Diag:approxdirectsum0}
\xi_0: 0 \xrightarrow{} K_0 \xrightarrow{g_0} M_0 \xrightarrow{f_0} Z,
\end{equation}
which yields $\mc{C}(M, \xi_0)=\Hom_R(M,\xi_0)$:
$$ 0 \ra \mc{C}(M, K_0) \ra \mc{C}(M, M_0) \ra\mc{C}(M,Z) \ra S_Z \ra 0.$$
Now decompose $K_0=M'_0 \oplus C_0$, where $C_0 \in \mc{D}$ and $M'_0 \in \mc{C}$. This decomposition is unique. This  means that $g_0$ is also decomposed into $M'_0 \oplus C_0 \xrightarrow{(g_{01},g_{02})} M_0$ with $g_{0i} \in {\bf J}_{\CM(R)}$ (by Lemma \ref{Lem:minimalres}). 
 Then 
 take a minimal $\mc{C}$-approximation of $C_0$, i.e., an exact sequence
$$ 0 \xrightarrow{} K_1 \xrightarrow{g_1} M_1 \xrightarrow{f_1} C_0$$
with $f_1$ right minimal. Then by Lemma \ref{Lem:exactapprox} $\Hom_R(M,-)$ makes this sequence also exact on the right, i.e.,
$$ 0 \xrightarrow{} \Hom_R(M,K_1) \xrightarrow{\Hom_R(M,g_1)} \Hom_R(M,M_1) \xrightarrow{\Hom_R(M,f_1)} \Hom_R(M,C_0) \xrightarrow{} 0.$$ 
Then a minimal right $\mc{C}$-approximation of $K_0$ is given as 

\begin{equation}  \label{Diag:approx1}
\begin{small}
\begin{tikzpicture}[description/.style={fill=white,inner sep=2pt}]    
    \matrix (m) [matrix of math nodes, row sep=3em,
    column sep=2.5em, text height=1.5ex, text depth=0.25ex]
    { 0 &  (M,K_1) &(M,M'_0 \oplus M_1) &  &  (M,M_0) &  (M,Z) & S_Z &0 \\
       &   &  & (M, M'_0 \oplus C_0) &  &  &  &\\ };
            \path[->,font=\scriptsize]
             (m-1-1) edge  (m-1-2)
              (m-1-2) edge node[auto] {$(M,\begin{pmatrix}0 \\ g_1 \end{pmatrix})$}  (m-1-3)
              (m-1-3) edge node[auto] {$(M,(g_{01}, g_{02} \circ f_1))$}  (m-1-5)
              (m-1-5) edge node[auto] {$(M,f_0)$}  (m-1-6)
              (m-1-6) edge  (m-1-7)
              (m-1-3) edge node[auto,swap] {$(M,\begin{pmatrix} Id_{M'_0} & 0 \\ 0 & f_1 \end{pmatrix})$}  (m-2-4)
               (m-2-4) edge node[auto,swap] {$(M,(g_{01},g_{02}))$}  (m-1-5)
              (m-1-7) edge  (m-1-8);

      \end{tikzpicture}
                \end{small}  
             \end{equation}

where we just write $( \cdot, \cdot)$ instead of $\mc{C}(\cdot, \cdot)$.
Note here that the map from $M'_0 \oplus M_1 \lra M_0$ is minimal and that by construction of \eqref{Diag:approx1} one gets an exact sequence
\begin{equation} \label{Diag:approxlong}
0 \xrightarrow{} (M,K_1) \xrightarrow{\begin{pmatrix}0 \\ (M,g_1) \end{pmatrix}} (M,M'_0 \oplus M_1) \xrightarrow{(M,g_{01}) \oplus (M,g_{02} \circ f_1) } (M,M_0)  \xrightarrow{(M,f_0)} (M,Z) \ra S_Z \ra 0.
\end{equation}

Here all maps are minimal. Now continue  by decomposing $K_1=M'_1 \oplus C_1$ and constructing a minimal $\mc{C}$-approximation $\xi_2$ of $C_1$.  Then one splices $\mc{C}(M, \xi_2)$ together with \eqref{Diag:approxlong} to get the next part of the minimal projective resolution. This method yields a minimal projective resolution of $S_Z$ (cf.~Lemma \ref{Lem:radicalminimal}). 
\end{Construction}

\begin{ex}
Let $R$ be a commutative complete local Cohen--Macaulay ring and consider $M:= R$. Then the construction above gives a minimal free resolution of the residue field $k$. The constructed minimal resolution is of finite length equal to $\dim R$ if and only if $R$ is regular.
\end{ex}

\section{Constructing minimal approximations and almost split maps via ladders} \label{Sec:tauladders}

 In this section $R$ will always be a noetherian commutative Henselian CM local ring with canonical module and with $\dim R \leq 2$ and $\mc{C}=\add(M)$ for some $\CM(R)$-module $M$. We will refer to these conditions as the Henselian setting.  We will use Iyama's notion of \emph{ladders} to construct right $\mc{C}$-approximations and right $\mc{C}$-almost split maps, in a similar context as in \cite{IyamaWemyss10a}, section 4. There, the special $\CM$-modules over two-dimensional quotient singularities are computed, that is, the (duals of) syzygies of $\CM$-modules.  
Therefore they construct an $\add(R)$-approximation via ladders in the factor category $\mc{D}=\CM(R)/[\add(R)]=\underline{\CM}(R)$ in Theorems 4.7 and 4.8 of loc.cit. In our case, in order to construct an $\add(M)$-approximation, we will consider the category $\mc{D}=\CM(R)/[\add(M)]$. Here we are not only interested in syzygies, but also in the actual ``approximation'' terms in $\add(M)$, since the goal is to explicitly construct a minimal projective resolution of a module over $A=\End_R M$.  In \cite{IyamaWemyss10a}, only the dual version of the result we need is proven, thus we will revisit the proof. 
We will work in the setting of $\tau$-categories, which were introduced and studied by Iyama in \cite{IyamaTau1}.

\begin{defi} \label{Def:taucat}
Let $\mc{C}$ be an additive category. Consider the properties: \\
(a) $\mc{C}$ is Krull-Schmidt. \\
(b) For any object $X \in \mc{C}$ there exists a complex 
\begin{equation} \label{Eq:rightTau} \tau X \xrightarrow{\nu_X} \vartheta X \xrightarrow{\mu_X} X,
\end{equation}
such that $\mu_X, \nu_X \in J_\mc{C}$ are right minimal morphisms and such that the sequences
$$\mc{C}(-,\tau X) \xrightarrow{\mc{C}(\nu_X)} \mc{C}(-,\vartheta X) \xrightarrow{\mc{C}(\mu_X)} J_\mc{C}(-,X) \lra 0$$
and
$$ \mc{C}(\vartheta X, -) \xrightarrow{\mc{C}(\nu_X)} J_\mc{C}(\tau X, -) \lra 0$$
are exact. \\
(c) For any object $X \in \mc{C}$ there exists a complex 
\begin{equation} \label{Eq:leftTau}
X \xrightarrow{\mu^-_X} \vartheta^- X \xrightarrow{\nu^-_X} \tau^-X,
\end{equation}
such that $\mu^-_X, \nu^-_X \in J_\mc{C}$ are left minimal morphisms and such that the sequences
$$\mc{C}(\tau^- X,-) \xrightarrow{\mc{C}(\nu^-_X)} \mc{C}(\vartheta^- X,-) \xrightarrow{\mc{C}(\mu^-_X)} J_\mc{C}(X,-) \lra 0$$
and
$$ \mc{C}(-,\vartheta^- X) \xrightarrow{\mc{C}(\nu^-_X)} J_\mc{C}(-,\tau^- X) \lra 0$$
are exact. \\
Then \eqref{Eq:rightTau} (resp.~\eqref{Eq:leftTau}) is called a \emph{right $\tau$-sequence (resp.~left $\tau$-sequence)}. $\mc{C}$ is called a \emph{right (resp.~left) $\tau$-category} if it satisfies conditions (a) and (b) (resp.~(a) and (c)). The category $\mc{C}$ is called a   \emph{$\tau$-category} if it satisfies (a)--(c).
\end{defi}

\begin{Bem}
If $\mc{C}$ is a (right) $\tau$-category and $\mc{D}$ is a full subcategory of $\mc{C}$, then by \cite[1.4]{IyamaTau2} the factor category $\mc{C}/[\mc{D}]$ is also a (right) $\tau$-category. Note that we only make use of right $\tau$-sequences, so for right $\tau$-categories the same proof as in \cite[1.4]{IyamaTau2} for $\tau$-categories holds. The $\tau$-sequence of an object $X \in \mc{C} \backslash \mc{D}$ is then given by its image in $\mc{C}/[\mc{D}]$, meaning, that if $\tau X \lra \vartheta X \lra X$ is a right $\tau$-sequence in $\mc{D}$, then if $\vartheta X$ is not contained in $\mc{D}$, $\overline{\tau X} \lra \overline{\vartheta X} \lra X$ is the corresponding $\tau$-sequence in the factor category, where $\overline{\phantom{a}}$ means that one removes direct summands contained in $\mc{D}$. If $\vartheta X = 0$, then the corresponding $\tau$-sequence is $0 \lra 0 \lra X$. Similar for left $\tau$-sequences.
\end{Bem}

 \begin{ex} \label{Ex:MCMtau}
If $R$ is an isolated singularity and complete local, then $\CM(R)$ for $\dim R \leq 2$ is a $\tau$-category, see \cite[2.2]{IyamaTau1}: such an $R$ can be seen as order over its Noether normalization, which is a complete regular local ring. \\ 
In our Henselian setting, we only assume that $\dim R \leq 2$ and that $R$ is an isolated singularity, CM Henselian local with canonical module, so need not be an order over a complete regular local ring. But then $\CM(R)$ is still a right $\tau$-category:  for any indecomposable non-projective $X \in \CM(R)$ there exists an almost split sequence, see \cite{Auslander84} or \cite[Cor.~13.9]{LeuschkeWiegand} for the Henselian setting,
$$ 0 \lra \tau X \lra \vartheta X \lra X \lra 0,$$
which plays the role of a right $\tau$-sequence. Any non-injective module $X$ has a left $\tau$-sequence. For the only projective indecomposable module in $\CM(R)$, that is, $R$ itself, one has
\begin{itemize}
\item If $\dim R=0$  then the fundamental sequence of $R$ is of the form
$$ 0 \lra 0 \lra \mf{m} \lra R \lra k \lra 0.$$
This gives a right $\tau$-sequence for $R$:
$$  0 \lra \mf{m} \lra R,$$
thus $\tau R=0$.  
\item If $\dim R=1$ ($R$ local and Henselian) then one has the same fundamental sequence, and $\tau R=0$. Note here that from the depth lemma it follows that $\mf{m}$ is a $\CM$-module over $R$, which might not be indecomposable.
\item If $\dim R=2$, then the fundamental sequence of $R$ is of the form (cf. \cite[(11.5)]{Yoshino90})
$$0 \lra \omega_R \lra \vartheta R \lra R \lra k \lra 0,$$
which gives the right $\tau$-sequence 
$$\omega_R \lra \vartheta R \lra R$$ 
for $R$.
Thus in this case $\tau R=\omega_R$ holds.
\end{itemize}
\end{ex}

\begin{defi}
Let $\mc{D}$ be a (right) $\tau$-category and $X  \in \mc{D}$. A \emph{right ladder} for $X$ is a diagram
\begin{equation}
\begin{array}{ccccccccc}
\cdots\xrightarrow{g_3}& Z_3&\xrightarrow{g_2}&Z_2&\xrightarrow{g^{}_1}&Z_1&\xrightarrow{g^{}_0}&Z_0=0\\
&\downarrow^{a^{}_3}&&\downarrow^{a^{}_2}&&\downarrow^{a^{}_1}&&\downarrow^{a^{}_0}&\\
\cdots\xrightarrow{f^{}_3}& Y_3&\xrightarrow{f^{}_2}&Y_2&\xrightarrow{f^{}_1}&Y_1&\xrightarrow{f^{}_0}&Y_0=X,
\end{array}
\end{equation}
such that each square commutes in $\mc{D}$, and moreover there exist objects $U_{n+1} \in \mc{D}$ and morphisms $h_{n} \in \mc{D}(U_{n+1},Z_n)$ such that 
$$ Z_{n+1} \oplus U_{n+1}  \xrightarrow{\begin{pmatrix}  -g_n & h_n \\ a_{n+1} & 0 \end{pmatrix}}  Z_{n} \oplus Y_{n+1} \xrightarrow{\begin{pmatrix}  a_n & f_n \end{pmatrix}} Y_n$$
is a right $\tau$-sequence for any $n \geq 0$. We say that the ladder \emph{terminates} if $Y_n=0$ for some $n \gg 0$. 
\end{defi}

In order to construct the terms of the ladder, one has a recursion formula in $\mc{K}_0(\mc{D})$, cf.~Thm.~4.8 (c) of \cite{IyamaWemyss10a} and \cite{IyamaTau1}, 7.1. Here $\mc{K}_0(\mc{A})$ denotes the Grothendieck group of an additive caterory $\mc{A}$. This group is a free abelian group generated by the isomorphism classes of the indecomposable objects in $\mc{A}$, when $\mc{A}$ has the Krull-Schmidt property, and the relation $[A' \oplus A']'=[A]+[A'']$, where $[A]$ denotes the class of an object $A \in \mc{A}$. However, the terms in the formula come from modules in $\CM(R)$, so one can also consider them in $\mc{K}_0(\CM(R))=\mc{K}_0(\mmod{R})$ (the last equality holds for $(R, \mf{m}, k)$ Henselian and Cohen--Macaulay, see e.g. \cite[13.2]{Yoshino90}). \\ 
Let now $\mc{D} \subseteq \CM(R)$ be a full subcategory. Note that each $X \in  \mc{K}_0(\mc{D})$ can be written uniquely in $\mc{K}_0(\mc{D})$ or $\mc{K}_0(\CM(R))$ as $X= X_+ - X_-$ for $X_+, X_- \in \mc{D}$ (and thus in $\CM(R)$) not having a common direct summand.  \\
In our context $\mc{C}=\add(M)$ and $\mc{D}$ will be the factor category $\CM(R)/[\add(M)]$.\\
Start with $X \in \CM(R)$ and its $\tau$-sequence $0 \ra \tau X \ra \vartheta X \ra X$. Set 
$$Y_0=X, \ Y_1=\vartheta X - M_\mc{C}(\vartheta X), \ P_1=M_\mc{C}(\vartheta X), $$
 where $M_\mc{C}(L)$ denotes the maximal direct summand of $L \in \CM(R)$ contained in $\mc{C}$. In order to get the terms in $\mc{K}_0(\mc{D})$ one just kills all terms contained in $\mc{C}$. Thus in $\mc{K}_0(\mc{D})$ one just has $Y_1=\vartheta X$ and $P_1=0$,  when one has to consider the $\tau$-sequence in $\mc{D}$. Now set 
 $$Z_0=0, \ U_0=0, \ P_0=0 \textrm{ and } Z_1'=\tau X.$$
In this first step one also sets $U_1=0$ and $Z_1=Z'_1$.
 Now we may define recursively for $n \geq 1$: 
\begin{equation} \label{Eq:recursion1}
  Y_{n+1} =(\vartheta Y_n - Z'_{n})_+ - M_\mc{C}((\vartheta Y_n - Z'_{n})_+), \quad
 P_{n+1}  =M_\mc{C}((\vartheta Y_n - Z'_{n})_+) 
 \end{equation}
 And moreover: 
 \begin{equation} \label{Eq:recursion2}
Z'_{n+1}  =\tau Y_n, \quad U_{n+1}=(\vartheta Y_{n} - Z'_{n+1})_-, \quad Z_{n+1}  =Z'_{n+1} - U_{n+1}. 
 \end{equation}
 Every expression is equivalent to one with only positive terms hence each equivalence class in $K_0$ is represented by an actual module  in $\CM(R)$.
 Note that $Z'_{n+1}=Z_{n+1} \oplus U_{n+1}$ just means that we decompose each $\tau Y_n$ into two direct summands such that for a map $(f_1, f_2): \tau Y_n=Z_{n+1}\oplus U_{n+1} \ra \vartheta Y_{n}$ the image of the component $f_2: U_{n+1} \ra \vartheta Y_n=Y_{n+1}\oplus P_{n+1} \oplus Z_n$ lies in $P_{n+1}\oplus Z_n$. \\
This can be easily deduced from the AR sequence: $U_{n+1}$ does not have a common summand (by construction) with $\vartheta Y_{n}$. Thus there is no irreducible map from $U_{n+1}$ to $Y_{n+1}$. But since there is a map from $U_{n+1}$ to $\vartheta Y_{n}$, its image has to lie in $P_{n+1} \oplus Z_n$.

 \begin{Bem} \label{Bem:laddertau}
If $\dim R \geq 3$, then there is no right $\tau$-sequence for $R$, so one cannot easily write down a recursion formula for the terms in the ladder as above. One might use higher AR-theory, as in \cite{IyamaHigher} in order to find  similar formulas.
 \end{Bem}

The next lemma is a slight generalization of the well known-fact about the finite length of the $\Hom$ in the stable category $\underline{\Hom}=\Hom_R/[\add(R)]$ between CM-modules of isolated singularities. This is the key observation to show that ladders terminate for rings of finite $CM$-type.

\begin{Lem} \label{Lem:Homfinitelength}
Let $R$ be a noetherian commutative  local ring with an isolated singularity. Let $M \in \CM(R)$ and $X \in \CM(R)$ and $Y \in \mmod{R}$. Then $(\Hom_R/[\add(M)])(X,Y)$ is a module of finite length over $R$.
\end{Lem}

This is Lemma 7.6 of \cite{AuslanderFunctors} with $\underline{\Hom}$ substituted by $\Hom_R/[\add(M)]$. However, since  it is not clear that all steps work in our setting, we will give the proof here: 

\begin{proof}
First note that $N$ is a module of finite length over $R$ if and only if $N_{\mf{p}}=0$ for all non-maximal prime ideals $\mf{p}$ . Secondly, $\Ext$ localizes, that is $\Ext^i_R(X,Y)_\mf{p} \cong \Ext^i_{R_\mf{p}}(X_\mf{p}, Y_{\mf{p}})$ for all $\mf{p} \in \Spec(R)$ and all $i \geq 0$. The second fact  implies that 
$$ ((\Hom_R/[\add(M)])(X,Y))_\mf{p} \cong \Hom_{R_\mf{p}}(X_\mf{p},Y_{\mf{p}})/[\add(M_{\mf{p}})](X_\mf{p},Y_{\mf{p}}).$$
Since $R$ is an isolated singularity, all localizations $R_\mf{p}$ with $\mf{p} \in \Spec(R)$ non-maximal are regular local rings. Thus, since $X$ and $M$ are $\CM$-modules, $X_\mf{p}$ and $M_\mf{p}$ are $R_{\mf{p}}$-projective, that is, $X_\mf{p} \cong R_\mf{p}^k$ and $M_{\mf{p}}\cong R_\mf{p}^l$ for some $k,l \in \N$. Hence any morphism $f: X_\mf{p} \lra Y_\mf{p}$ can be factored through $\add(M_\mf{p})=\add(R_{\mf{p}})$ (by e.g., $X_\mf{p}\cong R_\mf{p}^k \xrightarrow{id}  R_{\mf{p}}^k \xrightarrow{f} Y_\mf{p}$) and thus $\Hom_{R_\mf{p}}(X_\mf{p},Y_{\mf{p}})/[\add(M_{\mf{p}})](X_\mf{p},Y_{\mf{p}})=0=(\Hom_R/[\add(M)](X,Y))_\mf{p}$ for any non-maximal prime ideal $\mf{p}$. Then the general fact about finite length modules implies that $\Hom_R/[\add(M)](X,Y)$ is a module of finite length over $R$.
\end{proof}

\begin{Prop} \label{Prop:JacobianNilpotent}
Let $R$ be a local CM Henselian ring with an isolated singularity, $M \in \CM(R)$ and assume that $\mc{D}=\CM(R)/[\add(M)]$ is finite, that is, there is only a finite number of indecomposable objects in $\mc{D}$. Then there exists an $m \in \N$ such that ${\bf J}_\mc{D}^m=0$. 
\end{Prop}

\begin{proof}
We have to show that there exists an integer $m$ such that ${\bf J}^m_\mc{D}(X,Y)=0$ for any $X,Y \in \CM(R)$. Since ${\bf J}^m_\mc{D}(X,Y)$ is a submodule of $\Hom_R/[\add(M)](X,Y)$, it has finite length by  Lemma \ref{Lem:Homfinitelength}. Moreover, one has a chain ${\bf J}_\mc{D}(X,Y) \supset {\bf J}^2_\mc{D}(X,Y) \supset \cdots$ of submodules of $\Hom_\mc{D}(X,Y)$. Because of finite length, this chain has to stabilize, i.e., ${\bf J}^{m_{XY}}_\mc{D}(X,Y)=0$ for some $m_{XY} \gg 0$. Since there are only finitely many indecomposables in $\mc{D}$ and ${\bf J}^k(\bigoplus_i X_i, \bigoplus_j Y_j)=\bigoplus_{ij} {\bf J}^k(X_i,Y_j)$ for $X_i, Y_j$ indecomposable (see Lemma \ref{Lem:radicaldecomposition}), there exists an $m$ such that ${\bf J}_\mc{D}^m=0$.
\end{proof}

Essentially, the next theorem is Thm.~4.8 of \cite{IyamaWemyss10a} - however, Iyama and Wemyss work in the factor category,  and they prove the theorem for left ladders. Therefore, we are stating the theorem in our context and  show that a similar proof works. Moreover we prove that the constructed maps are indeed minimal right $\mc{C}$-approximations and -almost split maps.

\begin{Thm} \label{Thm:ladder}
Let $R$ be a local Henselian CM ring with canonical module and suppose that $R$ is an isolated singularity with $\dim R \leq 2$. Let $\mc{C}$ be a full subcategory of $\CM(R)$ 
 and let $\mc{D}$ be the factor category $\CM(R)/[\mc{C}]$. Take any $X \in \CM(R)$. Then there exists a commutative diagram (a ladder of $X$ in $\mc{D}$) such that
\begin{equation} \label{Eq:ladder}
\begin{array}{ccccccccc}
\cdots\xrightarrow{g_{31}}& Z_3&\xrightarrow{g_{21}}&Z_2&\xrightarrow{g^{}_{11}}&Z_1&\xrightarrow{g^{}_{01}}&Z_0=0\\
&\downarrow^{b^{}_3}&&\downarrow^{b^{}_2}&&\downarrow^{b^{}_1}&&\downarrow^{b^{}_0}&\\
\cdots\xrightarrow{f^{}_{31}}& Y_3&\xrightarrow{f^{}_{21}}&Y_2&\xrightarrow{f^{}_{11}}&Y_1&\xrightarrow{f^{}_{01}}&Y_0=X,
\end{array}
\end{equation}
where the objects $Y_n \in \mc{D}$, $P_n \in \mc{C}$ and $Z_n, U_n \in \CM(R)$ and $Z'_n=Z_n\oplus U_n$ and $Y'_n=Y_n \oplus P_n$ are determined by the recursion formula in \eqref{Eq:recursion1} and \eqref{Eq:recursion2}. Then there exist morphisms $b_n, g_n, f_n \in {\bf J}_{\CM(R)}$ with $b_n: Z_n \lra Y_n$, $g_n:=(g_{n1}, g_{n2}): Z_{n+1} \oplus U_{n+1} \lra Z_n$, $f_n:=(f_{n1},f_{n2}): Y_{n+1} \oplus P_{n+1} \lra Y_n$, $c_n:=(c_{n1},c_{n2}): Z_n \oplus U_n \lra P_n$, such that
\begin{equation} \label{diag:ex1}
Z'_{n+1} \xrightarrow{ \alpha_n} Z_{n} \oplus Y'_{n+1} \xrightarrow{ \beta_n} Y_{n}
\end{equation}
are right $\tau$-sequences of $Y_n$ in $\CM(R)$ for any $n \geq 0$. The morphisms here are 
$$\alpha_n: Z_{n+1} \oplus U_{n+1} \xrightarrow{ \begin{pmatrix} - g_{n,1} & - g_{n,2} \\ b_{n+1} & 0 \\ c_{n+1,1} & c_{n+1,2} \end{pmatrix}} Z_n \oplus Y_{n+1} \oplus P_{n+1}$$ 
and 
$$\beta_n: Z_n \oplus Y_{n+1} \oplus P_{n+1}\xrightarrow{\begin{pmatrix}  b_n & f_{n,1} & f_{n,2}\end{pmatrix}} Y_n.$$
If $\CM(R)$ is of finite $\CM$-type and $X\neq 0 \in \mc{D}$ (resp.~$X\neq 0 \in \mc{C}$), then the ladder computes a minimal right $\mc{C}$-approximation of $X$ (resp.~minimal right $\mc{C}$-almost split map). Namely, for some $n \gg 0$ there are maps $f, a$ with $a \in {\bf J}_{\CM(R)}$ such that 
\begin{equation} \label{Eq:exactapprox}
0 \lra Z_n \oplus \bigoplus_{i=1}^n U_i \xrightarrow{a} \bigoplus_{i=1}^n P_i \xrightarrow{f} X
\end{equation}
is exact and any other morphism (resp.~any other morphism in ${\bf J}_{\CM(R)}$) from some $T \in \mc{C}$ to $X$ factors through $f$. 
\end{Thm}

\begin{Bem}
For the proof of this theorem an \emph{extended ladder in $\CM(R)$} will be constructed. Then the statement of Theorem \ref{Thm:ladder} can be  given as follows: 
Let $R$ have the same properties as in Theorem \ref{Thm:ladder}, then there exists a commutative diagram (an extended ladder of $X$ in $\CM(R)$) such that 
 \begin{equation} \label{Eq:ladderMCM}
\begin{array}{ccccccccc}
\cdots\xrightarrow{g_3\oplus 1_{U_3}\oplus1_{U_2}\oplus 1_{U_1}}& A_3&\xrightarrow{g_2\oplus1_{U_2}\oplus 1_{U_1}}&A_2&\xrightarrow{g^{}_1\oplus1_{U_1}}&A_1&\xrightarrow{g^{}_0}&A_0=0\\
&\downarrow^{a_3}&&\downarrow^{a_n}& &\downarrow^{a_1}& &\downarrow^{a_0}&\\
\cdots\xrightarrow{f^{}_3\oplus 1_{P_3}\oplus1_{P_2}\oplus 1_{P_1}}& B_3&\xrightarrow{f^{}_2\oplus1_{P_2}\oplus 1_{P_1}}&B_2&\xrightarrow{f^{}_1\oplus 1_{P_1}}&B_1&\xrightarrow{f^{}_0}&Y_0=X,
\end{array}
\end{equation}
where $A_n=Z_n \oplus \bigoplus_{i=1}^nU_n$ and $B_n=Y_n \oplus \bigoplus_{i=1}^n P_i$ and the maps $a_n$ are explained below, and the remaining maps are defined as in theorem \ref{Thm:ladder}. \\
In addition, there is also the extended ladder of $X$ in $\mc{D}$, i.e., the extended ladder of $X \in \CM(R)$ with the $P_i$-terms killed: 
\begin{equation} \label{Eq:ladderextended}
\xymatrix{
\cdots \ar[r] &  Z_3 \oplus U_3 \oplus U_2 \oplus U_1 \ar[d]^{b^{}_3 \oplus 0 \oplus 0 \oplus 0} \ar[rr]^--{g_2 \oplus 1_{U_2}\oplus 1_{U_1}} & & Z_2 \oplus U_2 \oplus U_1 \ar[r]^---{g^{}_1\oplus 1_{U_1}}   \ar[d]^{b^{}_2 \oplus 0 \oplus 0} & Z_1 \oplus U_1\ar[r]^{g^{}_0}  \ar[d]^{b^{}_1 \oplus 0}&\ar[d]^{b^{}_0}    Z_0=0\\
\cdots\ar[r]^{f^{}_3}& Y_3 \ar[rr]^{f^{}_2}& &Y_2 \ar[r]^{f^{}_1 }&Y_1 \ar[r]^{f^{}_0 }&Y_0=X.
}
\end{equation}
\end{Bem}

\begin{proof}
We will construct the extended ladder of $X \in \CM(R)$. 
The maps $a_n$ in \eqref{Eq:ladderMCM} are given by diagonal matrices of the form $\begin{pmatrix}  b_n & 0 & \bf{0}_{1 \times (n-1)} \\ c_{n,1} & c_{n_2} & \bf{0}_{1 \times (n-1)} \\ & d_n & \end{pmatrix}$,
 where the maps $d_n: Z_n \oplus U_n \oplus \bigoplus_{k=1}^{n-1}U_k \lra  \bigoplus_{k=1}^{n-1}P_k$ can be inductively determined via  $d_1=0$, $d_2=(c_{11}g_{11}, c_{11}g_{12}, c_{12})$ and $d_n= \begin{pmatrix} c_{n-1} & \bf{0}_{1 \times (n-1)}  \\  d_{n-1}   & \end{pmatrix} \circ (g_n \oplus 1_{U_{n-1}} \oplus \cdots \oplus 1_{U_1})$ for $n \geq 3$. This means simply that $d_n$ is the composition of the maps from the ladder
 $$A_n \xrightarrow{g_n \oplus 1_{U_{n-1}} \oplus \cdots \oplus 1_{U_1}} A_{n-1} \xrightarrow{a_n} B_{n-1}=Y_{n-1} \oplus \bigoplus_{k=1}^{n-1}P_k$$
  restricted to $\bigoplus_{k=1}^{n-1}P_k$.   The image of \eqref{Eq:ladderMCM} in $\mc{D}$ is given by deleting the direct summands $P_i$ and the respective maps. The extended ladder (in accordance with the notation of Iyama--Wemyss)  in $\mc{D}$ is \eqref{Eq:ladderextended}. \\

Here  we start the actual proof: in (i) we show that \eqref{diag:ex1} is a right $\tau$-sequence, in (ii) that the squares in the  extended ladder commutes in $\CM(R)$, which implies that the images in $\mc{D}$ commute. Moreover it is shown that  each square is the direct sum of a right $\tau$-sequence and a split exact sequence. In (iii) we show that \eqref{Eq:exactapprox} is an exact sequence and in (iv) that it gives a minimal right $\mc{C}$-approximation of $X \neq 0 \in \mc{D}$ (resp.~a minimal right $\mc{C}$-almost split map of $X  \in \mc{C}$). \\

(i) The sequence \eqref{diag:ex1} is a right $\tau$-sequence: Let $Y_n, P_n, Z_n, U_n$ be defined as in the recursion formula in \eqref{Eq:recursion1} and \eqref{Eq:recursion2}. A right $\tau$-sequence of $Y_{n}$ is given as
$$ 0 \xrightarrow{} \tau Y_{n} \xrightarrow{} \vartheta Y_{n} \xrightarrow{} Y_{n}.$$
Then since $\tau Y_{n}=Z_{n+1}'=Z_{n+1} \oplus U_{n+1}$ and $\vartheta Y_{n}=Z_n \oplus Y_{n+1} \oplus P_{n+1}$ there exist maps in ${\bf J}_{\CM(R)}$ such that the sequence 
\begin{equation} \label{Eq:tausequence} Z_{n+1} \oplus U_{n+1} \xrightarrow{\begin{pmatrix}   - g_{n1} & - g_{n2} \\ b_{n+1} & 0 \\ c_{n+1,1} & c_{n+1,2}  \end{pmatrix}} Z_n \oplus Y_{n+1} \oplus P_{n+1} \xrightarrow{\begin{pmatrix} b_{n} & f_{n,1} & f_{n,2} \end{pmatrix}} Y_n
\end{equation}
is a $\tau$-sequence. \\
(ii) The maps in the ladder:  from the definitions  of $Y_n$ and  $Z_n$, it is clear that each square in the ladder commutes in $\mc{D}$, more concretely, from \eqref{Eq:tausequence} it follows by deleting the $P_i$ and $U_i$ that $b_{n} g_{n1}=f_{n1} b_{n+1}$. 
In the extended ladder in $\CM(R)$ we have chosen the $d_n$ so that the diagram commutes. \\
(iii) In order to see that \eqref{Eq:exactapprox} is exact, we follow the lines of the proof of Thm.~4.5 of \cite{IyamaWemyss10a}. Since we consider the dual situation, all arrows have to be reversed and instead of injective summands $I_i$ one has to consider $P_i \in \CM(R)/[\mc{D}]=\mc{C}$. \\
The construction is similar: let $B_n:=Y_n \oplus P_n \oplus \bigoplus_{k=1}^{n-1}P_k$ and $A_n:= Z_n \oplus U_n \oplus \bigoplus_{k=1}^{n-1}U_k$, so $a_n \simeq \begin{pmatrix} b_n & 0 & \bf{0}_{1 \times (n-1)} \\ c_{n,1} & c_{n,2} &\bf{0}_{1 \times (n-1)}\\ & d_n &  \end{pmatrix}$. \\

{\it Claim:} The sequences  $A_{n+1} \lra B_{n+1} \oplus A_n \lra B_n$ obtained from \eqref{Eq:ladderMCM} are isomorphic  to the direct sum (as complexes) of a right $\tau$-sequence of $Y_n$  and a split exact sequence of the form 
$$ 0 \ra U'_n \xrightarrow{\begin{pmatrix} 1 \\ 0 \end{pmatrix} } U'_n  \oplus P'_n \xrightarrow{(0 1)} P'_n \ra 0,$$
where  $P'_n=\bigoplus_{k=1}^{n} P_k$ and $U'_n=\bigoplus_{k=1}^{n} U_k$ for any $n$. \\

{\it Proof of claim:}  This can be carried out by an explicit calculation, which consists in writing the square of the ladder
$$\xymatrix{
A_{n+1}    \ar[r]_{g_n} \ar[d]_{a_{n+1}} & A_n \ar[d]^{a_n}   &  = &   Z'_{n+1}  \oplus U'_{n}  \ar[r] \ar[d]  &Z_n \oplus U'_n  \ar[d]    \\ 
B_{n+1}  \ar[r]^{f_n} & B_n &  &   Y_{n+1} \oplus P_{n+1} \oplus P'_n   \ar[r]& Y_n \oplus P'_n}$$
in form of an exact sequence and then comparing it to the direct sum of the right $\tau$-sequence for $Y_n$ and the above split exact sequence, which is of the form
\begin{equation} \label{Eq:directsum} 
Z'_{n+1} \oplus U'_n \xrightarrow{\phantom{aaaa}} Y_{n+1} \oplus P_{n+1} \oplus Z_n \oplus U'_n \oplus P'_n \xrightarrow{ \phantom{aaaa}} Y_n \oplus P'_n.
\end{equation}
Now one can construct an explicit isomorphism between the two.

(iv) Using (i)-(iii) and induction, we see that 
$A_n \simeq Z_n \oplus \bigoplus_{j=1}^n U_j$ and $B_n \simeq Y_n \oplus \bigoplus_{j=0}^n P_j$ in $\CM(R)$. In $\mc{D}$, one has $A_n \simeq Z_n \oplus \bigoplus_{j=1}^n M_{\mc{D}}(U_j)$, where $M_{\mc{D}}(U_j)$ denotes the maximal summand of $U_j$ contained in $\mc{D}$, and $B_n \simeq Y_n$. Since all the squares in the ladder \eqref{Eq:ladderMCM} commute and the first square forms an exact sequence $0 \lra Z_1 \lra Y_1 \oplus P_1 \lra X$ (the $\tau$-sequence of $X$),  the sequence 
\begin{equation} \label{Eq:exacttau}
0 \longrightarrow A_n \xrightarrow{a_n} B_n \xrightarrow{f=f_0f_1 \cdots f_{n-1}} X,
\end{equation}
where $f_i=(f_{i1},f_{i2}) \oplus 1_{P_{i-1}} \oplus \cdots \oplus 1_{P_1}$, is exact by a diagram chase. Clearly $f|_{Y_n}$ lies in ${\bf J}^n (Y_n, X) \subseteq {\bf J}^n_{\CM(R)}$. If $R$ is of $\CM$-finite type, it follows from Prop.~\ref{Prop:JacobianNilpotent} that $Y_n=0$ for $n \gg 0$. In this case we have an exact sequence in $\CM(R)$
$$ 0 \xrightarrow{} \bigoplus_{j=1}^n U_j \xrightarrow{a_n} \bigoplus_{j=1}^n P_j \xrightarrow{f} X,$$
with $P_j \in \mc{C}$ and $a_n \in {\bf J}_{\CM(R)}$, which shows that $f$ is minimal, cf.~Lemma \ref{Lem:minimalres}.  \\

(v) From the extended ladder in $\CM(R)$ we also get that \eqref{Eq:exactapprox} has the desired approximation properties. First note that any morphism $\gamma: T \lra N$ for some module $N \in \CM(R)$ and $T \in \mc{C}$ is contained in ${\bf J}(T,N)$ if $N$ is not isomorphic to a (sum of) direct summand(s) of $T$. Thus any morphism $\gamma: T \lra Y_i$, for any $Y_i \neq 0$ as above, $i \geq 1$, in the ladder will always be contained in ${\bf J}_{\CM(R)}$, since $Y_i \not \in \mc{C}$ by construction. \\
Suppose now that $\theta:T \ra X$ is in ${\bf J}(T,X)$.  
We will show inductively that there exists an $h: T \lra Y_{n+1} \oplus P_{n+1} \oplus \cdots \oplus P_1$ such that $\theta=f_0 \cdots f_n h$. It is clear that there exists an $h_1: T \lra Y_1 \oplus P_1$ such that $h=f_0 \circ h_1$ by the $\tau$-sequence property of $f_0$. Suppose we have constructed a $h_n=(h_{n1}, \ldots, h_{n,n+1}): T \ra Y_n \oplus \bigoplus_{i=1}^nP_i$ such that $f_0 \cdots f_{n-1} h_n=\theta$. Here $h_{n1}:T \ra Y_n$ is in ${\bf J}(T,Y_n)$ by the remark in the previous paragraph. Using that the $\tau$-sequence of $Y_n$ is given by \eqref{Eq:tausequence}, we obtain a morphism $\varphi=(\varphi_1, \varphi_2, \varphi_3): T \ra Z_n \oplus Y_{n+1} \oplus P_{n+1}$ such that 
\begin{equation} \label{Eq:tauid}
 b_n \varphi_1 + f_{n1} \varphi_2 + f_{n2} \varphi_3 = h_{n1}.
 \end{equation}
{\it Claim:} $(\varphi_2, \varphi_3, h_{n2}, \ldots, h_{n,n+1})$ can be taken as $h$. \\

To see this, compute 
$$f_0 \cdots f_n (\varphi_2, \varphi_3, h_{n2}, \ldots, h_{n,n+1})= f_0 \cdots f_{n-1}(f_{n1}\varphi_2 + f_{n2} \varphi_3, h_{n2}, \cdots, h_{n,n+1}).$$ 
Plugging in \eqref{Eq:tauid} yields $f_0 \cdots f_{n-1}(h_{n1} - b_n \varphi_1, h_{n2}, \cdots ,h_{n,n+1})$, which is equal to 
$$f_0 \cdots f_{n-1}(h_{n1}, h_{n2}, \ldots, h_{n,n+1}) - f_0 \cdots f_{n-1}(b_n\varphi_1, 0, \ldots 0).$$
 Moreover one has $(b_n\varphi_1, 0, \ldots, 0)=a_n(\varphi_1,0, \ldots, 0)$. Using that \eqref{Eq:exacttau} is exact in the middle, 
one finds that $f_0 \cdots f_{n-1}a_n(\varphi_1, 0, \ldots 0)=0$ and thus by induction hypothesis $h=(\varphi_2, \varphi_3, h_{n2}, \ldots, h_{n,n+1})$ satisfies $f_0 \cdots f_n h=f_0 \cdots f_{n-1} h_n=\theta$. 
\end{proof}

\begin{Bem}
Note that for this proof $M$ does not have to be a generator, that is, we do not require $R \in \add(M)$. The only thing needed is that the beginning of the ladder is a right $\tau$-sequence. If $X \neq R$ and $R \in \add(M)$, then \eqref{Eq:exacttau} will be exact on the right and all squares in the ladder are pullback/pushout squares. 
\end{Bem}

\section{Examples} \label{Sec:examples}

First we illustrate the ladder method for two endomorphism rings over the $E_6$-curve. In Section \ref{Sub:Leuschke} we compute the global dimension of Leuschke's NCRs \cite{Leuschke07}, as announced in \cite{DFI}. In \ref{Sub:An} we prove the formula for the global spectra of $A_n$-curves. In \ref{Sub:globalspecsimple} the results of the computation of global spectra for the ADE curve singularities are given together with a short description of our SAGE program. \\
Throughout this section we use Yoshino's \cite{Yoshino90} notation for the indecomposable $\CM$-modules over ADE curve singularities and refer to loc.~cit.~for the corresponding AR-quivers.

\begin{ex} \label{Ex:LeuschkeE6}
(The $E_6$-curve, Leuschke's NCR): Let $R=k[[x,y]]/(x^3+y^4)$. Let $M=R\oplus M_1 \oplus B$ and let $\mc{D}=\CM(R)/[\add(M)]$ and $\mc{C}=\add(M)$. Note that $\gl \End_R M$ has been computed with other methods in \cite[Example 5.12]{DFI}. In order to determine the global dimension of $A=\End_R M$ with ladders, let us compute the $\mc{C}$-almost split maps of the direct summands of $M$ and consequently the $\mc{C}$-resolutions of the appearing kernels: \\

(i) $R$: The fundamental sequence is 
$$ 0 \lra N_1 \cong \mf{m} \lra R,$$
and $R$ is the projective cover of the simple $k$. The $\mc{C}$-approximation of $N_1$ is given by the ladder:
\begin{equation} \label{Eq:ladder}
\begin{array}{cccccccccccccc}
B &\xrightarrow{} & X &\xrightarrow{} & B \oplus M_2 & \xrightarrow{}& X &\xrightarrow{}&B&\xrightarrow{}&M_1&\xrightarrow{}&0\\
\downarrow^{} & &\downarrow^{} &  &\downarrow^{} & &\downarrow^{}&&\downarrow^{}&&\downarrow^{}&&\downarrow^{}&\\
 M_1\oplus M_1 &\xrightarrow{} & A \oplus M_1 &\xrightarrow{}  &X \oplus M_1   &\xrightarrow{}& A \oplus M_2 &\xrightarrow{}&X&\xrightarrow{}&A&\xrightarrow{}&N_1.
\end{array}
\end{equation}
In terms of the recursion formula this reads as $Y_0=N_1$, $U_0=P_0=Z_0=0$. Then $Y_1=A$, $P_1=U_1=0$, $Z_1=M_1$. Further $Y_2=X$, $P_2=0$, $Z_2=B$ and $U_2=0$. Then $Y_3=A \oplus M_2$, $Z_3=X$ and $U_3=P_3=0$. In the fourth step: $Y_4=X$, $P_4=M_1$, $Z_4=B \oplus M_2$ and $U_4=0$. For $Y_5=A$, $P_5=0$, $Z_5= X$ and $U_5=0$, and the ladder finally terminates with $Y_6=0$, $P_6=M_1$, $Z_6=B$ and $U_6=0$. The right $\mc{C}$-almost split map of $R$ is thus
$$0 \lra B \lra M_1 \oplus M_1 \lra R,$$
which shows that $\pd S_R= \pd_{\End_R M}(\Hom_R(M,R)/{\bf J}(M,R))=2$. \\
(ii) $M_1$: the AR-sequence of $M_1$ is a $\mc{C}$-almost split map with kernel $N_1$. Since we have already computed the $\mc{C}$-approximation of $N_1$ in (i), the sequence
$$0 \lra B \lra M_1 \oplus M_1 \lra R\oplus B \lra M_1$$
yields that $\pd S_{M_1}=3$. \\
(iii) $B$: the extended ladder of $B$ looks as follows: 
\begin{equation} \label{Eq:ladder}
\begin{array}{cccccccccccc}
B &\xrightarrow{} & X & \xrightarrow{}& B \oplus M_2 &\xrightarrow{}&M_1 \oplus X &\xrightarrow{}& A &\xrightarrow{}&0\\
\downarrow^{} &  &\downarrow^{} & &\downarrow^{}&&\downarrow^{}&&\downarrow^{}&&\downarrow^{}&\\
 M_1 \oplus B &\xrightarrow{}  & A \oplus B   &\xrightarrow{}& X \oplus B &\xrightarrow{}& A \oplus M_2 \oplus B &\xrightarrow{}&N_1 \oplus X &\xrightarrow{}&B.
\end{array}
\end{equation}
The minimal $\mc{C}$-right almost split map of $B$ is thus
$$0 \lra B \lra B \oplus M_1 \lra B.$$
So $\pd S_B=2$. In total, $\gl \End_R M=3$.
\end{ex}

\begin{ex}
($E_6$ -- infinite global dimension) Let again $R=k[[x,y]]/(x^3+y^4)$. This time let $M=R\oplus B \oplus X$ and $\mc{C}=\add(M)$.  Using ladders, one sees that $\pd_A S_R=2$ and $\pd_AS_B=4$. However, $S_X$ is of infinite projective dimension:
again with ladders, one can show that the $\mc{C}$-almost split map of $X$ is 
$$ 0 \lra N_1 \oplus M_2 \lra B \oplus X \oplus R \lra  X.$$
 The $\mc{C}$-approximation of $N_1$ is
$$ 0 \lra B \lra X \lra N_1 \lra 0,$$
 and for $M_2$ it is just the AR sequence
$$ 0 \lra M_2 \lra X \lra M_2 \lra 0.$$
Hence the minimal $\mc{C}$-approximation of $N_1 \oplus M_2$ is
$$ 0 \lra B \oplus M_2 \lra X \oplus X \lra N_1 \oplus M_2 \lra 0.$$
But here we see that the $\mc{C}$-resolution will be periodic: $B$ is contained in $\mc{C}$, so it doesn't have to approximated further, but the kernel of the $\mc{C}$-approximation of $M_2$ will be again $M_2$. Thus $\pd_A S_X = \infty$ and $A$ is not of finite global dimension.
\end{ex}

\subsection{Leuschke's endomorphism rings for ADE curve singularities} \label{Sub:Leuschke}

\begin{Thm}
Let $R$ be a 1-dimensional ADE curve singularity, that is, a reduced complete local ring of type $A_n$ with $n \geq 1$, $D_n$ with $n \geq 4$, $E_6, E_7$ or $E_8$. Consider the family of rings $\{R^{(j_i)}_{i}\}$, where $R_0 = R$, $R_1 = R^{(1)}_1 \oplus \cdots \oplus R^{(n_1)}_1 := \End_{R} \mf{m}$, and the $R^{(j_i)}_i$ are the direct factors of all $\End_{R^{(l)}_{i-1}}\mf{m}_{R^{(l)}_{i-1}}$. Consider the longest chain $R_0 \subseteq R^{(j_1)}_1 \subseteq \cdots R^{(j_m)}_m = \widetilde R$ and set $A=\End_R \bigoplus_{i=0}^m R^{(j_i)}_i$.  Then 
$$\gl A= \begin{cases} 2 & \text{ if $R$ is of type $A_n$,} \\
    3&  \text{ else.}
    \end{cases}
    $$
\end{Thm}

\begin{proof} The proof consists of a case by case analysis: 
\begin{enumerate}[leftmargin=*]
\item $R\cong k[[x,y]]/(x^{n+1} + y^2)$, the $A_n$-case: see Section \ref{Sub:An}. 
\item{ $D_n$: start with the odd $n$: in \cite{DFI} it was computed that Leuschke's chain is given by $M=R\oplus X_1 \oplus A \oplus M_1 \oplus \cdots \oplus M_{\frac{n-3}{2}}$. In order to compute $\gl \End_RM$, we compute the minimal projective resolutions of the simples as above. Set again $\mc{C}:=\add(M)$.
The following exact sequences yield minimal $\End_R M$-projective resolutions of the corresponding simples: \\
(i) $R$: 
 $$0 \lra  M_1 \oplus A \lra X_1 \oplus X_1 \lra R \lra k \lra 0.$$ 
 (ii) $X_1$: 
 $$ 0 \lra M_1 \oplus A \lra X_1 \oplus X_1 \lra M_1 \oplus R \oplus A \lra X_1 \lra 0.$$
 (iii) $A$: 
 $$0 \lra M_1 \lra X_1 \lra A \lra 0.$$
 (iv) $M_i$: here one has to look at several cases: computing the ladder for $M_1$ yields $\mc{C}$-almost split map
 $$ 0 \lra A \oplus M_1 \lra X_1 \oplus M_2 \lra M_1 \lra 0.$$
 Note here that for $n=5$ the approximation is slightly different: $0 \lra A \oplus M_1 \lra M_1 \oplus X_1 \lra M_1 \lra 0$.
 For $1 < i < \frac{n-3}{2}$ one gets
 $$ 0 \lra M_i \lra M_{i-1} \oplus M_{i+1} \lra M_i \lra 0,$$
 and for $M_{\frac{n-3}{2}}$ one obtains
 $$0 \lra M_{\frac{n-3}{2}} \lra M_{\frac{n-3}{2}} \oplus M_{\frac{n-5}{2}} \lra M_{\frac{n-3}{2}} \lra 0.$$
 In all cases $\pd_{\End_R M} S_{M_i}=2$. \\
 Thus the maximum of the projective dimension of the simples is $3$ and hence $\gl \End_R M=3$. \\
 
 The even $D_n$'s are similar:  here $M=R \oplus X_1 \oplus M_1 \oplus \cdots \oplus M_{\frac{n-4}{2}} \oplus D_{-} \oplus D_+$ and the $\mc{C}$-approximations are either the same as for the $n$ odd case or as listed below for $n > 4$:\\ 
 (i) $M_{\frac{n-4}{2}}$: here a ladder yields the $\mc{C}$-approximation
 $$ 0 \lra M_{\frac{n-4}{2}} \lra M_{\frac{n-6}{2}} \oplus D_+ \oplus D_{-} \lra M_{\frac{n-4}{2}} \lra 0.$$
 (ii) $D_{+}$: here the approximation is
 $$0 \lra D_{-} \lra M_{\frac{n-4}{2}} \lra D_+ \lra 0.$$
 (iii) $D_{-}$ is similar:
 $$0 \lra D_{+} \lra M_{\frac{n-4}{2}} \lra D_{-} \lra 0.$$
 Again one sees that $\gl \End_R M =3$. \\
 For $n=4$, $M=R \oplus X_1 \oplus A \oplus D_+ \oplus D_{-}$ and the approximations slightly differ but the global dimension of $\End_R M$ is again $3$. One has the following $\mc{C}$-approximations: $0 \lra D_+ \oplus D_{-} \lra X_1^2 \lra R \lra k \lra 0$, $0 \lra D_+ \oplus D_{-} \lra X_1^2 \lra A \oplus D_+ \oplus D_{-} \lra X_1 \lra 0$, $0 \lra D_+ \oplus D_{-} \lra X_1 \lra A \lra 0$, $0 \lra D_+ \oplus A \lra X_1 \lra D_+ \lra 0$ and $0 \lra D_- \oplus A \lra X_1 \lra D_- \lra 0$. 
 }

\item $E_6$: see Example \ref{Ex:LeuschkeE6}. 

\item{ $E_7$: here $M=R \oplus M_1 \oplus Y_1 \oplus A \oplus D$. The following exact sequences yield minimal $\End_R M$-projective resolutions of the corresponding simples: 
\begin{equation*}
\begin{array}{ccccccccccc}
  0 &\lra &Y_1 &\lra &M_1^2 &\lra &R& \lra& k &\lra& 0. \\
 0&  \lra & Y_1& \lra & M_1^2& \lra &Y_1 \oplus R& \lra &M_1& \lra & 0. \\
0 &\lra & Y_1 & \lra & D \oplus A & \lra & M_1 & \lra & Y_1& \lra  & 0. \\
0 & \lra & Y_1 & \lra & M_1^3 & \lra & Y_1& \lra & A&  \lra & 0.\\
& & 0 & \lra &A  & \lra &  Y_1 & \lra  & D &  \lra & 0.
\end{array}
\end{equation*}
Again $\gl \End_R M = 3$. 
}

\item{ $E_8$: here $M=R \oplus M_1 \oplus A_1 \oplus A_2$ and the approximations of the direct summands are as follows: 
\begin{equation*}
\begin{array}{ccccccccccc}
 0 &\lra & A_1 &\lra &M_1^2 &\lra& R& \lra & k &\lra&  0.\\
0 &\lra  &A_1 &\lra &M_1^2 &\lra &A_1 &\lra & M_1 & \lra &0. \\
& & 0 &\lra &A_1 &\lra &A_2 \oplus M_1 &\lra&  A_1&\lra & 0. \\
& &0 & \lra &A_2 &\lra & A_1 \oplus A_2 &\lra&   A_2 &\lra& 0.
\end{array}
\end{equation*}
} 
Again $\gl \End_R M = 3$.
\end{enumerate}
\end{proof}

\subsection{Global spectrum of $A_{n}$-curves} \label{Sub:An}

In \cite{DaoFaber} the global spectrum of $A_n$-curve singularities is computed by a general fact about triangulated categories. Here we give a computational proof using ladders:

\begin{Thm} \label{Thm:specAn}
(1) $n$ odd: Let $R=k[[x,y]]/(y^2+x^n)$, where $n=2k+1$. Then the global spectrum $\gs_{\CM(R)}(R)=\{1,2\}$. In particular the endomorphism rings of finite global dimension are of the form $A_{i_0}=\End_R(\bigoplus_{i=i_0}^k I_i)$, where $I_i=(y,x^i)$ are the indecomposable $\CM(R)$-modules (note: $I_0=R$!) and $i_0 \geq 0$. Here 
$$\gl A_{i_0}=\begin{cases} 1 \ \text{ if } i_0=k \\
                                             2 \ \text{ else.}
                                             \end{cases}
                                             $$
                                             
 (2) $n$ even: Let $R=k[[x,y]]/(y^2+x^n)$, where $n=2k$. Then the global spectrum $\gs_{\CM(R)}(R)$ is $\{1,2,3\}$. The indecomposables in $\CM(R)$ are the $I_i$, $i=0, \ldots, k$ and the two smooth irreducible components $D_+$ and $D_-$. There are three types of endomorphism rings $A=\End_R M$ of finite global dimension: \\
 (i) $\add(M) \subseteq \add(D_-, D_+)$, then $\gl A=1$. \\
 (ii) $M=\bigoplus_{i=i_0}^k I_i \oplus D_+ \oplus D_-$, with $0 \leq i_0 \leq k$ then $\gl A=2$. \\
 (iii) $M= \bigoplus_{i=i_0}^{i_1} I_i \oplus D_+$, with $i_0 \leq i_1 \leq k$  (or symmetrically with $D_-$ instead of $D_+$), then $\gl A=3$.                      
\end{Thm}

\begin{proof}
(1) Note that we can always assume that $M$ is a generator, that is, $R \in \add(M)$. This follows from the observation that $\End_R(I_1)=\End_R(\mf{m} )\cong k\{x,y\}/(y^2+x^{n-2})$ and that all $I_i$ with $i \geq 1$ are modules over $\End_R(\mf{m})$. So one can just substitute $R$ with $\End_R(\mf{m})$. Inductively it follows that if $i_0$ is that smallest index such that $I_{i_0}$ is a direct summand of $M$, then one may consider $M$ as an $I_{i_0} \cong \End_R(I_{i_0}) \cong k\{x,y\}/(y^2+x^{n-2i_0})$-module such that $M$ is then a generator. \\
Suppose now that $M=\bigoplus_{i=0}^l I_i \oplus \bigoplus_{j=m}^{l'}I_j$ for some $0 \leq l < m-1 \leq k-1$. We compute the projective resolution of the simple for $I_l$: the AR-sequence for $I_l$ is
\begin{equation}
0 \lra I_l \lra I_{l-1} \oplus I_{l+1} \lra I_l \lra 0.
\end{equation}
So for the ladder we get $Y_0=I_l$, $Y_1=I_{l+2}$, $P_1=I_{l-1}$, $Z_1=I_l$, $U_1=0$. The next terms are $Y_2=I_{l+2}$, $P_2=0$, $Z_2=I_{l+1}$, $U_2=0$. Iterating this process, until we reach $m$ gives: $Y_{m-l}=0$, $P_{m-l}=I_{m}$, $Z_{m-l}=0$ and $U_{m-l}=I_{m-1}$.  So a right almost $\add(M)$-split sequence of $I_l$ is given as:
$$ 0 \lra I_{m-1} \lra I_l \oplus I_m \lra I_l \lra 0.$$
Now construct an $\add(M)$-approximation of $I_{m-1}$:
In the ladder we have $Y_0=I_{m-1}, Y_1=I_{m-2}, P_1=I_m$ and $Z_1=I_{m-1}$, $U_1=0$. One easily sees the structure of the ladder from the next step: $Y_2=I_{m-3}, P_2=0$ and $Z_2=I_{m-2}$, $U_2=0$. Thus we get $Y_{m-l-1}=0$, $P_{m-l-1}=I_{l}$ and $U_{m-l-1}=I_{l+1}$ and $Z_{m-l-1}=0$. Hence
$$ 0 \lra I_{l+1} \lra I_l \oplus I_m \lra I_{m-1} \lra 0$$
is the desired $\add(M)$-approximation. Now the ladder for the kernel $I_{l+1}$ has a very similar structure and one gets its $\add(M)$-approximation
$$0 \lra I_{m-1} \lra I_m \oplus I_l \lra I_{l+1}.$$
But here we already constructed the approximation of the kernel. The minimal projective resolution of the simple $S_{I_l}= \Hom_R(M, I_{l})/ {\bf J}(M,I_{l})$ is thus given as
$$ \cdots \lra P_{I_m} \oplus P_{I_l} \lra P_{I_m} \oplus P_{I_l} \lra P_{I_l} \lra S_{I_l} \lra 0.$$
Hence $\gl \End_R M= \infty$. \\
It remains to show that the only possibility for finite global dimension is $l'=k$. Therefore suppose that $l' < k$, that is, $M$ is of the form $\bigoplus_{i=0}^{l'} I_{i}$. Compute the minimal projective resolution of the simple $S_{I_{l'}}$ similarly as before: the right $\add(M)$-almost split sequence of $I_{l'}$ is
$$ 0 \lra I_{l' +1} \lra I_{l'} \oplus I_{l' -1} \lra I_{l'} \lra 0. $$
The $\add(M)$-approximation of the kernel $I_{l'+1}$ is
$$ 0 \lra I_{l' +1} \lra I_{l'} \oplus I_{l'} \lra I_{l'+1} \lra 0. $$
Thus the (infinite) minimal projective resolution of $S_{I_{l'}}$ is
$$ \cdots \lra P_{I_{l'}} \oplus P_{I_{l'}} \lra P_{I_{l'}} \oplus P_{I_{l'}}  \lra P_{I_{l'}} \oplus P_{I_{l' -1}} \lra P_{I_{l'}} \lra S_{I_{l'}} \lra 0.$$
Hence we have shown that the only possibility that $\gl \End_R M < \infty$ is $M=\bigoplus_{i=i_0}^k I_{i}$. Here, if $i_0 < k$, then $M$ is a representation generator of the $A_{n-2i_0}$ singularity and thus $\gl \End_RM =2$ (either by the Iyama--Leuschke theorem or by direct calculation with ladders). If $i_0=k$, then $R$ is regular and $M=I_k \cong \widetilde R$. This yields $\gs \End_R M=1$.  \\

(2) The proof for the odd $A_n$-curves is similar. Here we will not compute all the steps in the ladders, only show which steps are necessary. The reader is invited to compute the ladders himself.
\begin{enumerate}[leftmargin=*]
\item It is easy to see that if $M=D_+$, $M=D_-$ or $M=D_+ \oplus D_-$, then $\gl A=1$. This covers all modules $M$ such that no $I_i$ is contained in $\add(M)$.
\item If at least one $I_i \in \add(M)$, we may assume that $I_0=R \in \add(M)$: choose $I_{i_0}$ with $i_0$ minimal. Since $I_{i_0} \cong R/(y^2 + x^{n+1-2i_0})$ an $A_{n-2i_0}$-singularity, all other modules in $\add(M)$ will be modules over $I_{i_0}$. Thus we can consider the endomorphism ring over the finite extension $I_{i_0}$ of $R$. 
\item There is no gap between the $I_i$'s: using ladders show that if $M$ is of the form $R \oplus \cdots \oplus I_{i_1} \oplus I_{i_2} \oplus \cdots$ with $i_1 < i_2-1$, then $\gl \End_R M=\infty$. This implies, that if $A$ has finite dimension and one is not in the case $(i)$, then $M$ is of the form $\bigoplus_{i=0}^{i_1}I_i$, $0 \leq i_1 \leq k$ direct sum with possibly summands in $\{ D_+, D_-\}$.
\item If $M$ is of the form $M=\bigoplus_{i=0}^{i_1}I_i$ with $0 \leq i_1 \leq k$ or $M=\bigoplus_{i=0}^{i_1}I_i\oplus D_+ \oplus D_-$, with $0 \leq i_1 < k$, then $\gl A= \infty$. This can again be shown using ladders. 
\item Now assume wlog $D_+ \in \add(M)$. If $M$ is of the form $M=\bigoplus_{i=0}^{i_1}I_i\oplus D_+$,  with $0 \leq i_1 \leq k$, then $\gl A=3$. Using ladders, one can show that the simples $S_{I_i}$ have projective dimension $2$, whereas the simple $S_{D_+}$ has projective dimension $3$.
\item Finally, if $\add(M) = \CM(R)$, then $M$ is a representation generator and the Iyama--Leuschke theorem shows that $\gl A=2$.
\end{enumerate}
\end{proof}

It is possible to determine the number of endomorphism rings of torsion-free modules of finite global dimension of $A_n$-singularities (up to Morita-equivalence):

\begin{Cor}
For an $A_{2k}$-singularity $R$ there is one endomorphism ring $A$ with global dimension $1$: $A=\End_R(\widetilde R) \cong \widetilde R$, the normalization. There are $k$ endomorphism rings $A$ with global dimension $2$: $A=\End_R(\bigoplus_{i=i_0}^k I_i)$, $0 \leq i_0 \leq k$, where $I_0=R$. \\
For an $A_{2k+1}$ -singularity $R$ there are three endomorphism rings $A$ with global dimension $1$: $A\cong D_+$, $A \cong D_{-}$ and $A \cong D_+ \oplus D_{-}$. There are $k+1$ endomorphism rings with global dimension $2$: $A \cong \End_R(\bigoplus^{k}_{i=i_0} I_i \oplus D_+ \oplus D_-)$,  $0 \leq i_0 \leq k$, where $I_0=R$. There are $k^2 + 3k +2$ endomorphism rings with global dimension $3$: $A=\End_R( \bigoplus_{i=i_0}^{i_1} I_i \oplus D_+)$, with $i_0 \leq i_1 \leq k$  (or symmetrically with $D_-$ instead of $D_+$)
\end{Cor}

\begin{proof}
The numbers are obvious, except for the global dimension $3$ case for $A_{2k+1}$-singularities. This case is a counting argument: by Theorem \ref{Thm:specAn} (2) (iii) we have to count the number of modules of the form $\bigoplus_{i=i_0}^{i_1} I_i \oplus D_+$, $i_0 \leq i_1 \leq k$. For $k=0$ there is $1$ such module, namely $R \oplus D_+$. Because of the symmetry $D_+$ / $D_-$ there are in total $a_0=2$ endomorphism rings of global dimension $3$ for the $A_1$-curve. For $k=1$ there are  $2$ modules of the form $I_{i_0} \oplus D_+$ and one of the form $R \oplus I_1 \oplus D_+$, thus in total one gets $a_1=2(2+1)=6$. The same counting can be done for $A_{2k+1}$: there are $k+1$ modules of the form $I_{i_0} \oplus D_+$, $k$ modules of the form $I_{i_0} \oplus I_{i_0+1} \oplus D_+$, thus $k-l+1$ modules of the form $\bigoplus_{i=i_0}^{i_0+l}I_i \oplus D_+$, for $0 \leq l \leq k$. This yields 
$$a_k=2(1 + \cdots + (k+1))=(k+1)(k+2).$$
\end{proof}

\subsection{Global spectra of the ADE curves}  \label{Sub:globalspecsimple}

\begin{Thm} \label{Thm:globalspecsimple} Let $R$ be a 1-dimensional ADE curve singularity, i.e., a reduced complete local ring of type $A_n, D_n, E_6, E_7$ or $E_8$. Then
$$\gs(R)= \begin{cases} \{ 1,2 \} \text{ if $R$ is of type $A_{2n}$, $n \geq 1$}, \\
\{1,2,3\} \text{ if $R$ is of type $A_{2n+1}$, $n \geq 1$},\\
\{1,2,3,4\} \text{ if $R$ is of type $D_4$, $D_5$ or $E_6$,} \\
\{1,2,3,4,5\} \text{ if $R$ is of type $D_n$, $6 \leq n \leq 13$,}\\
\{1,2,3,4,5,6\} \text{ if $R$ is of type $E_7$ or $E_8$.}
\end{cases}$$
If $R$ is of type $D_n$, $n \geq 14$, then $\{1, \ldots, 5 \} \subseteq \gs(R)$.
\end{Thm}

\begin{proof}
The assertions for type $A_n$ are the content of Theorem \ref{Thm:specAn}. For the other cases we used a  SAGE program, which will be described below. It computes, for a given $\CM$-module $M,$ the length of the projective resolutions of the simples of $\End_RM$ using Theorem \ref{Thm:ladder} and Construction \ref{Const:projres}. Here it is sufficient to take only basic endomorphism rings, since any other endomorphism ring is Morita equivalent to a basic one, see Lemma \ref{Lem:MoritaReduced}. Since any ADE curve is of finite $\CM$-type, one can compute all possibilities. \\
It is easy to see that $\{1,\ldots,5\}$ is contained in $\gs(R)$ for $R$ of type $D_n$, $n\geq 14$: the curve singularities $A_{3} \vee L$ and $A_4 \vee L$ (see \cite{FruehbisKrueger} for notation) both have global spectrum $\{1,\ldots,5\}$.  Since any curve singularity $A_{n-2i-1} \vee L$ with coordinate ring $R'$ is an overring of $R$ (see Examples \ref{Ex:Dnodd} and \ref{Ex:Dneven}) and $\gs(R') \subseteq \gs(R)$ by Lemma \ref{Lem:gsOverring}, the assertion follows.
\end{proof}

\begin{Conj}
Let $R$ be of type $D_n$, $n \geq 14$. Then $\gs(R)=\{1,\ldots,5\}$.
\end{Conj}

\subsubsection{The Program}
In this section we describe the implementation in Sage. The code is available at: \url{http://kappa.math.unb.ca/research/brandoncode.html}\\

The input consists of the AR-quiver of a ring $R$ of finite $\CM$-type and of $\dim R \leq 2$ and a module $M \in \CM(R)$. For curves, one adds a formal zero-module which is treated as $\tau R$ and $\tau(0)=R$ to ensure $\tau$ is defined on every module so we obtain a translation quiver  (cf.~Example \ref{Ex:MCMtau}). The AR-quiver is encoded as a directed graph. The translation $\tau$ is encoded as an adjecency matrix and $M$ as a list of integers, called $S$, which represent the indecomposables of $M$.  
The program computes a list of the projective modules in the minimal projective resolutions of the simples of $\End_R M$ and thus the global dimension of $\End_R M$ as the maximum length of these resolutions. \\
The construction follows the description of the algorithm in section 4 of \cite{WemyssGL}. We indicate the notions of Theorem \ref{Thm:ladder}. Of central importance for this program is a directed graph which constitutes a finite portion of the universal cover of the AR quiver for the ring. The graph will be broken up into levels, with  the vertices in level $n$ being the predecessors of those in level $n - 1$. To keep track of the correspondences between the modules of the AR quiver and the vertices of the universal cover, we use a list which contains the number in the quiver of the module corresponding to the vertex numbered $v$ in the universal cover in position $v$. This graph is constructed step-by-step, starting with level $0$, as the ladder proceeds. 

Moreover, the following is used by the program:
\begin{itemize}
\item{Various lists to keep track of data concerning the vertices of the graph, such as their associated numbers and the level of the graph at which they are found.}
 \item{A number to keep track of the current level of the graph (the level whose vertices are being numbered).}
\item{A vector representing the kernel of the approximation, with the multiplicity in the kernel of the module numbered $n$ in the $n^{th}$ position (initially these are all set to zero)}.
\item{A vector representing the middle term of the approximation, similar to the above.}
\end{itemize}

 The algorithm consists of the following steps: \\
{\it (1) Constructing a right $\add(M)$-almost split map resp.~right $\add(M)$-approximation:}  We describe the process of associating numbers to the vertices of the graph.  To start with, the initial vertex, that is, the module to be approximated, ($Y_0$ in Thm.~\ref{Thm:ladder}) of the graph is numbered 1, and a new level is added to the graph, consisting of that vertex's predecessors. When a new vertex is first added to the graph, its number is set as zero. After a new level is added, it becomes the current level, and all vertices at that level are numbered. The number of a vertex $v$ is calculated by adding up the numbers of the successors  of $v$ in the universal cover, and subtracting the number associated to $\tau{}^{-1}v.$  If any of these vertices are not in the finite part of the graph which is constructed by the program, their numbers are assumed to be zero, since the only vertices that can have nonzero numbers are those from which there exists an edge to the initial vertex, and these are also the only vertices that can be included in the finite graph. For the purpose of this computation, vertices corresponding to modules in the set $S$ or the zero module and vertices whose numbers are negative are treated as if their numbers were zero. After a vertex is numbered (say its number is $k$), we first check to see if $k$ is negative; if so, the multiplicity of the corresponding module in the kernel is increased by $-k$ (this corresponds to the $U_n$ in Thm. \ref{Thm:ladder}). If $k$ is non-negative, we then check to see if $v$ corresponds to a module in the set $S$; if so, the multiplicity of the corresponding module in the middle term is increased by $k$ (corresponds to $P_n$ in Thm. \ref{Thm:ladder}). 

After all vertices at the current level are numbered, we check whether all vertices at both the current and previous levels are either numbered zero, or satisfy the criteria to have their numbers treated as zero when numbering other vertices. If this is the case, the step is finished (since any further vertices would always be numbered zero) and the program returns a list consisting of the vector representing the kernel, followed by the vector representing the middle term and goes to step (2). If not, a new level consisting of the current level's predecessors is added to the graph, which becomes the current level, and we start with step (1) again. \\
{\it (2) Approximating the kernel:} We examine the kernel $K$ computed in (1). The algorithm terminates under two conditions:
\begin{itemize}
\item{If all the indecomposable summands of $K$ are in the set $S$, then the approximation sequence is complete, and a list of vectors representing all terms in the sequence is returned.}
\item{If the set of indecomposable summands of $K$ which are not in $S$  is the same as the set of indecomposable summands for the kernel of another approximation in the same sequence,  then we know the sequence will ultimately turn out to be infinite, so a null value is returned.}
\end{itemize}
If one of the two conditions is met, the program returns either a list of terms in the sequence, or a null value yielding to an infinite resolution as above. \\
Otherwise follow construction \ref{Const:projres}: each module in the kernel which is not in $S$ is removed from the kernel (i.e. has its multiplicity changed to zero) and approximated with step (1). These approximations are then added together, with the same multiplicities as the modules had in the kernel, to produce an approximation for the part of the kernel which is not in $S$. This increases the length of the approximation sequence by one. We then examine the kernel of this new approximation sequence, and repeat step (2) until one of the two stopping conditions listed above is satisfied. \\

Now repeat the above procedure for all members of $S$. The global dimension of the endomorphism ring is equal to the length of the longest such sequence (which may be infinite). \\

We include a table of the occurring global dimensions and the involved modules of Thm.~\ref{Thm:globalspecsimple}, so that the complexity of the algorithm is more visible (see Table \ref{Tab:globaldimcurves}).

\begin{small}
\begin{center}
\begin{table}[h]
  
\begin{tabular}{|c|cccccc|cc|c|}
\hline
Singularity & \multicolumn{6}{c}{\# of subsets with $\gl$} & Finite $\gl$ & Infinite  & Total \\
\hline \hline
            & $1$ & $2$ & $3$ & $4$ & $5$  & $6$ &   & & \\ 
\hline
$E_6$  & $1$ & $13$  & $34$ & $4$ & $0$ & $0$ & 52 &  75 &  $2^7-1$ \\ 
$E_7$   & $3$ & $80$ & $7,638$ &	$6,933$ &	$486$ & $8$ & $15,148$ 	& $17,619$ 	&$2^{15}-1$ \\
$E_8$ & $1$	&$94$  & 	$24,614$ & 	$26,479$	& $2,500$ & $48$	&  $53,736$	&	$77,335$ &	$2^{17}-1$ \\
\hline
$D_n$, $n$ even & \multicolumn{8}{l}{}  & \\
\hline
$D_4$ &$7$ &	$28$	&$207$ &	$90$ &	$0$ & 0	& $332$	 &	$179$ & 	 $2^9-1$ \\
$D_6$ &$7$ &	$73$ & $2,416$ & $1,713$ & $66$ & 0	& $4,275$ & $3,916$  & $2^{13}-1$ \\
$D_8$ & $7$ &	$146$ &	$25,601$ & $26,743$ & $1,458$ & 0 & $53,955$   & $77,116$  &	$2^{17}-1$ \\
$D_{10}$ & $7$ & $253$ & $265,602$  &	$389,942$ &$23,422$ & 0   &	$679,226$  &	$1,417,925$	& $2^{21}-1$ \\
$D_{12}$ & $7$ & $400$ & $2,745,634$ & $5,449,152$ & $353,644$ & $0$ &$8,548,837$ & $25,005,594$ & $2^{25} - 1$ \\
\hline
$D_n$, $n$ odd & \multicolumn{8}{l}{} &   \\
\hline
$D_5$ &$3$ &	$20$ &	$95$ &	$26$ &	$0$ &	$0$ & $144$  &	$111$	& $2^8-1$ \\
$D_7$ & $3$ &	$58$ &	$1,164$ &  $555$ &	$16$ & $0$ & $1,796$  & $2,299$  &	$2^{12}-1$ \\
$D_9$ & $3$ &	$122$ &	$12,541$ & 	$9,527$ &	$382$ & $0$ &	$22,575$  & $42,960$  & $2^{16}-1$ \\
$D_{11}$ & $3$ & $218$ & $130,672$ &$146,418$ & $6,778$ & $0$ & $284,089$  & $764,486$  & $2^{20}-1$ \\
$D_{13}$ & $3$ & $352$ & $1,352,109$ & $2,113,324$ & $109,690$ & 0 & $3,575,478$ & $13,201,737$ & $2^{24} - 1$ \\
\hline
$A_{2k+1}$ & $3$ & $k+1$ & $k^2+3k+2$ & $0$ & $0$ & $0$ & $k^2+4k+6$ &   & $2^{k+3}-1$ \\
$A_{2k}$ & $1$ & $k$ & $0$ & $0$ & $0$ & $0$ & $k+1$ &   & $2^{k+1}-1$ \\
\hline
\end{tabular} \vspace{0.1cm}
\caption{Number of endomorphism rings of certain global dimensions} \label{Tab:globaldimcurves}
\end{table}
\end{center}
\end{small}
Recall that we exclude the case $M =0$.

\subsection{Global spectra of some surface singularities}

In this section we give examples of the computation of the global spectra for some surface singularities of finite $\CM$ type.  The method and algorithms are essentially the same as for the case of curves, but somewhat simpler since $\tau(R)=\omega_R$ and so the Auslander-Reiten 
translate is $\neq 0$ for all $\CM$ modules.
The surface singularities of finite $\CM$ type are of the form $k[x,y]^G$ where 
$G \subset \GL_2(k)$, with the characteristic of $k$ not dividing $|G|$. We use Brieskorn's classification of these singularities with the notation of Riemenschneider \cite{Riemenschneider}. We compute the global spectra for four examples.  We list the groups $G$ and the invariant rings $k[x,y]^G$ for each example.  We also show the AR quiver for two of the examples, which is equal to the MacKay graph.  The AR translate is indicated by a dashed arrow.  The results for the global spectra are summarized in a table. 

\begin{enumerate}[leftmargin=*]
\item $C_{8,5}$: 
$$G=\left\langle \left( \begin{array}{ccc}
 \zeta_8 & 0  \\
0 & -\zeta_8
 \end{array} \right) \right\rangle,  \quad \quad k[x,y]^G=k[x^{8},x^{3}y,xy^3,y^8]$$

$$
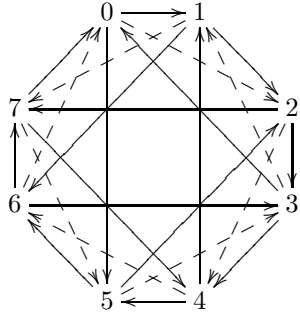
\xymatrix{
 & 0 \ar[r] \ar[ddd] \ar@{-->}[rrd] & 1 \ar[dr] \ar[ddll] \ar@{-->}[dll]& \\
7 \ar[ur] \ar[ddrr] \ar@{-->}[ddr]& & & 2 \ar@{-->}[ddl] \ar[d] \ar[lll] \\
6 \ar[u] \ar[rrr] \ar@{-->}[uur] & & & 3 \ar[dl] \ar[uull] \ar@{-->}[uul] \\
 & 5 \ar[ul] \ar[uurr] \ar@{-->}[rru] & 4 \ar[l] \ar@{-->}[llu]\ar[uuu] & \\
 }$$

\item $C_{16,9}$: 
$$
G=\left\langle
 \left( \begin{array}{ccc}
\zeta_{16} & 0  \\
0 &  -\zeta_{16}
 \end{array} \right) \right\rangle,  \quad \quad k[x,y]^G=k[x^{16},x^7y,x^5y^3,x^3y^5,xy^7,y^{16}]$$

\item $D_{5,3}$: 
$$G=\left\langle \left( \begin{array}{ccc}
\zeta_3^{2} & 0  \\
0 & \zeta_3
 \end{array} \right) ,
 \left( \begin{array}{ccc}
0 & \zeta_8  \\
\zeta_8 & 0 
 \end{array} \right) \right\rangle,$$ $$k[x,y]^G=k[ x^{4}y^{4}, x^{12} - y^{12},  x^{7}y + xy^{7}, x^{9}y^{3} - x^{3}y^{9}]$$

$$
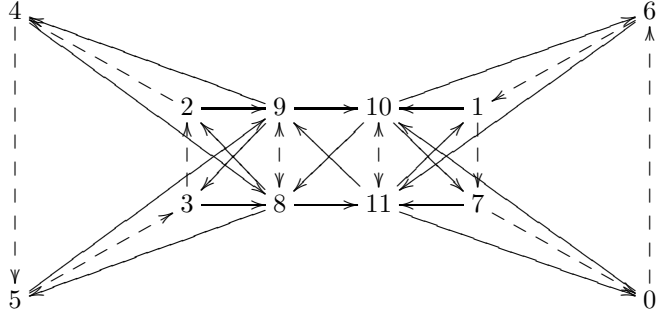
\xymatrix{
4 \ar[rrrdd] \ar@{-->}[ddd] &  &    &    &      &    &  & 6 \ar[llldd]  \ar@{-->}[lld]  \\
   &  & 2 \ar[r]  \ar@{-->}[llu] & 9 \ar[ld] \ar[r]  \ar[lllu] \ar@{-}@{<-->}[d] & 10 \ar[dl] \ar[dr] \ar[rrru] \ar@{-}@{<-->}[d] & 1 \ar[l] \ar@{-->}[d] &  &    \\
   &  & 3 \ar[r] \ar@{-->}[u] & 8 \ar[llld] \ar[r] \ar[lu] & 11 \ar[lu] \ar[ru] \ar[rrrd]  & 7 \ar[l] \ar@{-->}[rrd]&  &    \\
5 \ar[rrruu] \ar@{-->}[rru]&  &    &    &      &    &  & 0 \ar[llluu] \ar@{-->}[uuu] \\
}$$

\item $D_{7,5}$: 

$$G=\left\langle\left( \begin{array}{ccc}
\zeta_5 & 0  \\
0 & \zeta_5^{-1}
 \end{array} \right), \left( \begin{array}{ccc}
0 & \zeta_8  \\
\zeta_8 & 0 
 \end{array} \right)\right\rangle,$$

$$k[x,y]^G=k[x^{4}y^{4}, x^{20} -y^{20}, x^{11}y - xy^{11}, x^{13}y^{3} + x^{3}y^{13}]$$
\end{enumerate}

\begin{tabular}{ |c|r|   r|  r|  r|  r|  r | r | r  |  r | r | r | r|c| }
\hline Singularity & \multicolumn{9}{c}{\# of subsets with gl. dim} & Finite (total) & Infinite & Total \\  
\hline \hline
 & 1 & 2 & 3 & 4 & 5 & 6 & 7 & 8 & 9 & & &  \\
\hline
$C_{8,5}$ & 0 & 1 & 72 & 8 & 8 & 0 & 0 & 0 & 0 &  89 & 166 & $2^{8}-1$  \\
\hline
$C_{16,9}$ & 0 & 1 & 10,488 &	23,032 & 10,144 & 2,304 & 336 & 16 & 16 & 46,337 & 19,198 & $2^{16}-1$ \\
\hline
$D_{5,3}$ & 0 & 1 & 732 &  340 & 280 & 0 & 0 & 0 & 0 & 1,353 & 2,742 & $2^{12}-1$ \\
\hline
$D_{7,5}$ & 0 & 1 & 7,568 & 5,968 & 3,548 & 0 & 0 & 0 & 0 & 17,085 & 48,450 & $2^{16}-1$ \\
\hline \end{tabular}

\subsection{Inifinite $\CM$-type}  \label{Sub:infiniteCM}

It is natural to ask whether ladders can also be used to compute global dimensions of endomorphism rings if $R$ is not of finite $\CM$-type. We give two examples before we study the question more generally.

\begin{ex} \label{Ex:Kahn}
 Let $R=\C[[x,y]]/(y(y-x^2)(y-ax^2))$ with $a \neq 0,1$ be the coordinate ring of an $\widetilde E_8$-singularity. Then the $AR$-quiver is completely known, see \cite[Thm. 7.9 and Cor. 7.11]{KahnDiss} and also \cite{Dieterich} - we use the notation of \cite{KahnDiss} here: it consists of a disjoint union of so-called tubes $\mc{T}_{r,d}$  of type $D_4$, where $(r,d) \neq (1,1)$ and $r,d \in \Z$ and $1 \leq r \leq d \leq 2r$ and the tube $\mc{T}_{(1,1)}$ containing $R$. Below is a picture of $\mc{T}_{(1,1)}$: \\
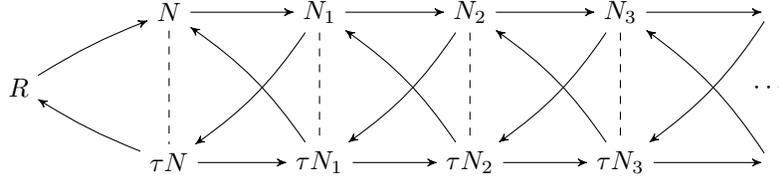
\[
\begin{tikzpicture}
\node at (4,0) {\begin{tikzpicture} 
\node (C1) at (0,0)  {$R$};
\node (C2) at (2,1)  {$N$};
\node (C3) at (2,-1)  {$\tau N$};
\node (C4) at (4,1)  {$N_1$};
\node (C5) at (4,-1)  {$\tau N_1$};
\node (C6) at (6,1)  {$N_2$};
\node (C7) at (6,-1)  {$\tau N_2$};
\node (C8) at (8,1)  {$N_3$};
\node (C9) at (8,-1)  {$\tau N_3$};
\node (C10) at (10,0)  {$\cdots$};
\node (C11) at (10,1)  {};
\node (C12) at (10,-1)  {};

\draw [->,bend left=5,looseness=1,pos=0.5] (C1) to node[]  {} (C2);
\draw[-,bend left=0,looseness=1,pos=0.5, dashed]  (C2) edge [] node[left] {} (C3);
\draw [->,bend left=5,looseness=1,pos=0.5] (C3) to node[] {} (C1);

\draw [->,bend left= 0, looseness=1,pos=0.5] (C2) to node[]  {} (C4);
\draw [->,bend left= 0, looseness=1,pos=0.5] (C4) to node[]  {} (C6);
\draw [->,bend left= 0, looseness=1,pos=0.5] (C6) to node[]  {} (C8);
\draw [->,bend left= 0, looseness=1,pos=0.5] (C8) to node[]  {} (C11);

\draw [->,bend left= 0, looseness=1,pos=0.5] (C3) to node[]  {} (C5);
\draw [->,bend left= 0, looseness=1,pos=0.5] (C5) to node[]  {} (C7);
\draw [->,bend left= 0, looseness=1,pos=0.5] (C7) to node[]  {} (C9);
\draw [->,bend left= 0, looseness=1,pos=0.5] (C9) to node[]  {} (C12);

\draw [->,bend left=10,looseness=1,pos=0.5] (C4) to node[] {} (C3);
\draw [->,bend left=10,looseness=1,pos=0.5] (C6) to node[] {} (C5);
\draw [->,bend left=10,looseness=1,pos=0.5] (C8) to node[] {} (C7);
\draw [->,bend left=10,looseness=1,pos=0.5] (C11) to node[] {} (C9);

\draw [->,bend left=-10,looseness=1,pos=0.5] (C5) to node[] {} (C2);
\draw [->,bend left=-10,looseness=1,pos=0.5] (C7) to node[] {} (C4);
\draw [->,bend left=-10,looseness=1,pos=0.5] (C9) to node[] {} (C6);
\draw [->,bend left=-10,looseness=1,pos=0.5] (C12) to node[] {} (C8);

\draw[-,bend left=0,looseness=1,pos=0.5, dashed]  (C4) edge [] node[left] {} (C5);
\draw[-,bend left=0,looseness=1,pos=0.5, dashed]  (C6) edge [] node[left] {} (C7);
\draw[-,bend left=0,looseness=1,pos=0.5, dashed]  (C8) edge [] node[left] {} (C9);

\end{tikzpicture}};
\end{tikzpicture}
\] 

Since $R$ is complete, $\CM(R)$ is a $\tau$-category and the $\tau$-sequences are just the AR-sequences and we can apply Theorem \ref{Thm:ladder} to compute ladders. Consider $M=R \oplus N \oplus \tau N$. Since $\tau N \cong \mf{m}$, the $\tau$-sequence for $R$ is $0 \lra \tau N \lra R$, which is already a right $\add(M)$-almost split map. The $\tau$-sequence for $\tau N$ is $0 \lra N \lra N_1 \lra \tau N \lra 0$. The ladder for $\tau N$ looks as follows: 
\begin{equation} \label{Eq:ladder}
\begin{array}{cccccccccccccc}
\cdots &\xrightarrow{} & Z_{2i}=\tau N_{2i-1} &\xrightarrow{} & Z_{2i}=\tau N_{2i-1} & \xrightarrow{}& \cdots &\xrightarrow{}&\tau N_1 &\xrightarrow{}&N&\xrightarrow{}&0\\
 & &\downarrow^{} &  &\downarrow^{} & & &&\downarrow^{}&&\downarrow^{}&&\downarrow^{}&\\
\cdots &\xrightarrow{} & Y_{2i+1}=N_{2i+1} &\xrightarrow{}  & Y_{2i}=\tau N_{2i}   &\xrightarrow{}& \cdots &\xrightarrow{}&\tau N_2&\xrightarrow{}&N_1&\xrightarrow{}&\tau N.
\end{array}
\end{equation}

Here any $Y_i, Z_i \neq 0$ and thus the ladder does not compute a  right $\add(M)$-almost split map for $\tau N$. 
\end{ex}

\begin{Bem}
Note that the underlying AR quivers  of the previous examples are of type $A_\infty$ which is the case for most of curves with infinite $\CM$-type, as explained in \cite{Dieterich85}.
\end{Bem}

\begin{ex}(Nonisolated singularity) \label{Ex:Ainfinity}
Let $R=\C[[x,y]]/(x^2)$, then $R$ is of countable $\CM$-type (cf. \cite{BuchweitzGreuelSchreyer}): the indecomposables are ideals of the form $M_j=(x,y^j), j=0, 1, \ldots$ and $M_\infty \cong k[[y]]$. The AR-quiver is well-known of   type $A_\infty$  and looks as follows:
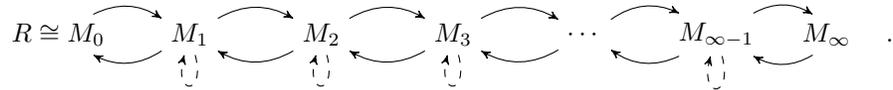
\[
\begin{tikzpicture}
\node at (4,0) {\begin{tikzpicture} 
\node (C1) at (0,0)  {$R \cong M_0$};
\node (C2) at (1.75,0)  {$M_1$};
\node (C3) at (3.5,0)  {$M_2$};
\node (C4) at (5.25,0)  {$M_3$};
\node (C5) at (7,0)  {$\cdots$};
\node (C6) at (8.75,0)  {$M_{\infty-1}$};
\node (C7) at (10.50,0)  {$M_\infty$ \quad .};

\draw [->,bend left=30,looseness=1,pos=0.5] (C1) to node[]  {} (C2);
\draw [->,bend left=30,looseness=1,pos=0.5] (C2) to node[] {} (C1);
\draw[->]  (C2) edge [in=50,out=-50,loop below ,looseness=10,pos=0.5, dashed] node[left] {} (C2);

\draw [->,bend left=30,looseness=1,pos=0.5] (C2) to node[]  {} (C3);
\draw [->,bend left=30,looseness=1,pos=0.5] (C3) to node[] {} (C2);
\draw[->]  (C3) edge [in=50,out=-50,loop below ,looseness=10,pos=0.5, dashed] node[left] {} (C3);

\draw [->,bend left=30,looseness=1,pos=0.5] (C3) to node[]  {} (C4);
\draw [->,bend left=30,looseness=1,pos=0.5] (C4) to node[] {} (C3);
\draw[->]  (C4) edge [in=50,out=-50,loop below ,looseness=10,pos=0.5, dashed] node[left] {} (C4);

\draw [->,bend left=30,looseness=1,pos=0.5] (C4) to node[]  {} (C5);
\draw [->,bend left=30,looseness=1,pos=0.5] (C5) to node[] {} (C4);

\draw [->,bend left=30,looseness=1,pos=0.5] (C5) to node[]  {} (C6);
\draw [->,bend left=30,looseness=1,pos=0.5] (C6) to node[] {} (C5);
\draw[->]  (C6) edge [in=50,out=-50,loop below ,looseness=10,pos=0.5, dashed] node[left] {} (C6);

\draw [->,bend left=30,looseness=1,pos=0.5] (C6) to node[]  {} (C7);
\draw [->,bend left=30,looseness=1,pos=0.5] (C7) to node[] {} (C6);
\end{tikzpicture}}; 
\end{tikzpicture}
\] 

The fundamental sequence of $R$ is $0 \lra M_1 \lra R$, the AR-sequences are $M_j \lra M_{j-1} \oplus M_{j+1} \lra M_j$ for $j\geq 1$, and $M_{\infty}$ does not have an AR-sequence. Take now $M=M_0 \oplus M_3$. Then one can use ladders to compute a minimal resolution of $S_{M_0}$ over $\End_R M$: the minimal right $\add(M)$-almost split map of $R \cong M_0$ is the fundamental sequence. A ladder yields the $\add(M)$-approximation of $M_1$: $0 \lra M_2 \lra M_0 \oplus M_3 \lra M_1 \lra 0$. Similar the $\add(M)$-approximation of $M_2$ is given as $0 \lra M_1 \lra M_0 \oplus M_3 \lra M_2 \lra 0$. Thus we see that $\Hom_R(M,-)$ into the resulting long exact sequence
$$ \cdots \lra M_0 \oplus M_3 \lra M_0 \oplus M_3 \lra M_0$$
yields that $S_{M_0}$ has infinite projective dimension over $\End_R M$. This example shows that one can sometimes use ladders to compute $\add(M)$-resolutions even though $\CM(R)$ is not a $\tau$-category. However, as in Example \ref{Ex:Kahn} a ladder will not compute a right $\add(M)$- almost split map of $M_3$. Moreover, if one computes the ladder for $M_i$, $i \geq 4$, then $Z_j=\bigoplus_{k=0}^jM_{i-j-1+2k}$ and $Y_j=\bigoplus_{k=0}^j M_{i-j+2k}$ for $j \leq i-2$ and $Z_{j}=\bigoplus_{k=0}^{i-4}M_{j-i+7+2k}$ and $Y_{j}=\bigoplus_{k=0}^{i-4}M_{j-i+8+2k}$ for $j \geq i-3$. Since any $M_j \neq 0$, it follows that $Y_j \neq 0$ and the ladder does not terminate.
\end{ex}

Let now $R$ be as in Thm.~\ref{Thm:ladder}, that is, $R$ is a local Henselian CM ring with canonical module and $\dim R \leq 2$ and $R$ is an isolated singularity. If $R$ is of finite $\CM$-type and $M \in \CM(R)$, then this theorem shows that a ladder computes an $\add(M)$-approximation/$\add(M)$-almost split map for any $X \in \CM(R)$, i.e., in the recursion formula there is some index $N$ such that $Y_n$ and $Z_n$ are $0$ for  $n \geq N$. However, as we have seen in in Examples  \ref{Ex:Kahn} and \ref{Ex:Ainfinity}, this need not hold for infinite $\CM$-type. So an evident question is:

\begin{Qu}
Let $R$ be as in Thm.~\ref{Thm:ladder} and let  $M \in \CM(R)$. Does a ladder compute an $\add(M)$-approximation for any $N \in \CM(R)$ if and only if $R$ is of finite $\CM$-type? 
 \end{Qu}

It is easy to see that if the AR quiver $\Gamma$ of $\CM(R)$ has at least two components $\Gamma_1 \coprod \Gamma_2$ then the question has a positive answer: take e.g.~any indecomposable $M \in \CM(R)$ such that $[M] \in \Gamma_1$. Then the recursion formula of Theorem \ref{Thm:ladder} applied to any indecomposable $N \in \CM(R)$, such that $[N] \in \Gamma_2$, does not compute an $\add(M)$-approximation for $N$. \\
 For infinite type AR quivers, there are structure theorems, see \cite{Dieterich85} and \cite{HappelPreiserRingel} for the case of Artin algebras. However, in general it is not clear, whether the AR-quiver consists of more than one component and how the AR translation acts on the quiver. But we can determine a special case:

\begin{Prop} \label{Prop:infiniteGorenstein}
Let $R$ be as in Thm.~\ref{Thm:ladder} and also assume that $R$ is Gorenstein of Krull-dimension $2$  and of infinite $\CM$-type. Let $M \in \CM(R)$. Then the ladder construction of Thm.~\ref{Thm:ladder} will not yield an $\add(M)$-approximation ($\add(M)$-split map) for  some  $X \in \CM(R)/[\add(M)]$ ($X \in \add(M)$). 
\end{Prop}

\begin{proof}
We show that  there always exists a component of the stable AR-quiver of type $A_\infty$. Similar to  Example \ref{Ex:Ainfinity}, one sees that if a component of the stable AR-quiver of $R$ is of $A_\infty$-type,  one can find a direct summand of $M$ such that its ladder does not terminate. \\
If $R$ is Gorenstein of Krull-dimension $d$, then the AR-translation $\tau$ of any $M \in \underline{\CM}(R)$ is given by $\tau M=\Hom_R(\mathrm{syz}_d(\mathrm{tr} M), R)$, where $\mathrm{tr}M$ denotes the Auslander-transpose of $M$. For $d=2$ it follows that $\tau M \cong M$ (from e.g., \cite{Auslander86}, Cor. 6.2.).  \\
The ranks of $\CM(R)$-modules are unbounded by the Brauer--Thrall theorem (see \cite{Yoshino87}, Prop.~4.1). Then arguing as in the proof of  loc.~cit, Prop.~4.1 or from Theorem 3 of \cite{Dieterich85}, it follows that the stable AR-quiver is of type $A_\infty$.
\end{proof}

\begin{Bem}
For $\dim R=1$ the translation $\tau$ is in general not the identity, even in the Gorenstein case: one has $\tau M \cong \mathrm{syz}_1 M$ (see Lemma (9.8) of \cite{Yoshino90}). Moreover, in order to apply structure theorems as Dieterich's \cite{Dieterich85}, one needs the existence of a periodic $\tau$-orbit. It would be interesting to study questions about existence and cardinality of periodic $\tau$-orbits for non-Gorenstein rings.
\end{Bem}

\section{Centres of endomorphism rings} \label{Sec:centre}

Here we determine the centres of the endomorphism rings of $\CM$ (i.e., torsion free) modules over curve singularities of finite $\CM$-type. This helps to speed up the determination of global spectra: if the centre of an endomorphism ring of an $R$-module is an overring of $R$, then in some cases   one already knows the global spectrum of this ring and does not have to compute projective resolutions again. 

\begin{defi}(see \cite{Yoshino90})
Let $(R, \mf{m}, k)$ be a one-dimensional local analytic $k$-algebra, with $k$ of characteristic $0$ and suppose that $R$ is reduced. Denote by $\widetilde R$ the integral closure of $R$ in its quotient ring. A ring $R'$ \emph{(birationally) dominates} $R$ if 
$$R \subseteq R' \subseteq \widetilde R.$$
\end{defi}

\begin{Lem} \label{Lem:gsOverring}
Let $R \subseteq R'$ be a ring birationally dominating $R$. Then $\gs R' \subseteq \gs R$.
\end{Lem}

\begin{proof}
Take any $\CM(R')$-module $M$. Then $\End_{R'} M \cong \End_R M$ and thus if $\gl \End_{R'} M =d$, it follows that $d \in \gs R$.
\end{proof}

\begin{Thm}  \label{Thm:centres}
Let $R$ be reduced noetherian ring, and $M$ be a faithful torsion free module over $R$. Then the centre of $\End_R M$ is the largest finite extension of $R$, which is an $R$-algebra and over which $M$ is a module (i.e., the largest ring $S$ such that $R \subseteq S \subseteq Q(R)$ and $SM=M$).
\end{Thm}

\begin{proof}
Let $Q=Q(R)$ be the quotient ring of $R$.
Suppose that $M$ is torsion free. We have the following diagram:
\begin{equation*}
\begin{tikzpicture}[description/.style={fill=white,inner sep=2pt}]    \label{Diag:trace1}
    \matrix (m) [matrix of math nodes, row sep=3em,
    column sep=2.5em, text height=1.5ex, text depth=0.25ex]
    { R &   \End_R M \\
 Q &   \End_R(M) \otimes_R Q=\End_Q(M \otimes Q) \\ };
            \path[->] (m-1-1) edge (m-1-2);
             \path[->] (m-2-1) edge (m-2-2);
             \path[right hook->] (m-1-1) edge (m-2-1);
             \path[left hook->] (m-1-2) edge (m-2-2);
      \end{tikzpicture}
             \end{equation*}
 From this it follows that
 $$Z(\End_R(M))=Q \cap \End_R(M).$$

Consider $R'= Q \cap \End_RM$. Then $R'$ is clearly contained in $\End_R M$, and since the latter is finitely generated as an $R$-module, also $R'$ is finitely generated. From this and the fact that $R' \subseteq Q$ it follows that $R' \subseteq \widetilde R$.
Thus $Z(\End_RM)$ is contained in $\widetilde R$. \\
Let $R \subset S \subset \widetilde R$ be a ring such that $M$ is a module over $S$. Then there is a map from $S$ to $\End_RM = \End_S M$ and $S \subseteq Z(\End_R M)=Z$, so also the largest $S$ with this property is contained in $Z$. On the other hand, $M$ is a module over the centre of its endomorphism ring, so $Z$ has to be contained in the largest integral extension $S$ such that $M$ is an $S$-module. \end{proof}

If $R$ is irreducible, then the rings dominating $R$ are totally ordered: $R \subseteq R_1 \subseteq \cdots \subseteq R_n \cong \widetilde R$.   The maximal element of this chain is always the normalization. In this case, if $M=\bigoplus_{i=1}^k M_i$ is a direct sum of indecomposable torsion-free $R$-modules, then $Z(\End_R(M))=\max_i \{ R_i:$ each $M_j$ is a module over $R_i\}$. If $R$ is reducible one only has a partial order with index set $I$ so $R \subseteq R_i \subseteq R_j$, for 
$i \leq j  \in I$,
then still  $Z(\End_R(M))=\prod \max \{ R_i :$ each $M_j$ is a module over $R_i\}$. 
 \\

Before we compute the centres for the curve singularities of finite $\CM$-type, a few general observations: First we have to determine the rings dominating $R$, that is, the rings $R'$, satisfying $R \subseteq R' \subseteq \widetilde R$. If $R$ is a ADE curve singularity, then by \cite[Satz 1]{GreuelKnoerrer} there are only finitely many such $R'$. Moreover, an indecomposable $\CM(R)$-module can only be isomorphic to an overring of $R$ if it has rank $1$ on $R$. If  $R=k\{x,y\}/f_1 \cdots f_k$ is reducible, one has to consider a rank vector $(a_1, \ldots, a_k)$, where $a_i$ denotes the rank of a module $N$ on the component $R/(f_i)$. So by rank $1$ module $N$  in this case we mean rank vector with $a_i=0$ or $1$ for all $i$ and at least one $a_i \neq 0$. If we have a ring $R'$ dominating $R$ so $R \subseteq R' \subset Q(R)$, then since the rank vectors of both $R$ and $Q(R)$ are $(1, \ldots, 1)$ we can can conclude that the rank of $R'$ is also $(1,\ldots,1)$.  One can compute the rank of a $\CM$-module $N$ by using that the rank of $N$ on each component $R/(f_i)$ is equal to the power with which $f_i$ appears in the determinant of the matrix factorization of $N$, see \cite{Eisenbud80}. \\
If $R' \supseteq R$ is a ring dominating $R$, then one can use trace ideals to determine whether a $\CM(R)$-module $N$ is still defined over $R'$: by definition, the trace ideal $\tr_R(N)$ of an $R$-module $N$  is given as the image of $\tr: N^*\otimes_{\End_R N}N \rightarrow   R$, where $\tr(f \otimes n)=f(n)$, see e.g., \cite{AuslanderGoldman}. One can show that $\tr_R(N)$ is the ideal generated by the entries of the matrix factorization belonging to $\mathrm{syz}(N)$, see \cite{Vasconcelos98}. One can show that if $\tr_R(N) \not \subseteq \tr_R(R') \cong \mc{C}_{R'/R}$, then $N$ is not a module over $R'$. Here $\mc{C}_{R'/R}$ denotes the conductor of $R'$ into $R$. 

If $M$ does not have full support, then one can still determine the centre of $\End_RM$ with the following lemma, whose proof is straightforward:

\begin{Lem} \label{Lem:support}
Let $R$ be a commutative ring, let $M,N$ be $R$-modules and let $I$ be an ideal in $R$ satisfying $I \subseteq \mathrm{Ann}_R(M)$ and $I  \subseteq \mathrm{Ann}_R(N)$. Then $\Hom_R(M,N)=\Hom_{R/I}(M,N)$. \\
In particular, if $R=k\{x,y\}/(f_1 \cdots f_k)$ and $M$ is an $R$-module annihilated by $f_1 \cdots f_l$ for some $l < k$, then $\End_RM=\End_{R/(f_1 \cdots f_l)}(M)$.
\end{Lem}

\subsection{Irreducible ADE curves}

\begin{ex} $A_n$, $n$ even: Here $R=k\{x,y\}/(y^2+x^{n+1})$. Any basic $\CM$-module is of the form $\bigoplus_{i \in I}I_i$, where $I \subseteq \{1, \ldots, n\}$. Denote by $i_0$ the minimal integer contained in $I$. Since $I_{i_0} \cong \End_R(I_{i_0}) \cong k\{x,y\}/(y^2+x^{n-2i_0})$ is dominating $R$ and is also the largest integral extension over which $I_{i_0}$ is a module, it follows that $Z(\End_R \bigoplus_{i\in I}I_i)\cong I_{i_0} \cong k\{x,y\}/(y^2+x^{n-2i_0})$.
\end{ex}

\begin{ex} $E_6$: Here $R=k\{x,y\}/(x^3+y^4)$. 
We first compute the ranks of the indecomposable torsion-free modules over $R$ from Yoshino's list of matrix factorizations, see \cite{Yoshino90}, p.79: 
the indecomposable $\CM$ modules of rank one over $R$ are $R, M_1, N_1, M_2, B$  Since for an overring $R \subseteq R'$ one must have $\End_R R' \cong R'$ it follows from $\End_R N_1 \cong M_1$ that $N_1$ is not isomorphic to an overring of $R$. Thus there are two rings dominating $R$:
$$R \subseteq  M_1 \subseteq  M_2 \subseteq  \widetilde R \cong B.$$
The following table shows a list of the indecomposables, their trace ideals, the largest ring over which they are defined and their singularity type (if they are isomorphic to a ring -- see \cite{FruehbisKrueger} for classification of space curve singularities):
 \begin{table}[h!]
  \centering
  \begin{tabular}{ |c|c|c|c|c|}
    \hline
     Module & rank & trace & largest ring & singularity type \\ 
   \hline \hline 
    $R$ & $1$ & $(1)$ & $R$ & $D_n$ \\ 
       $M_1$ & $1$ & $(x,y)$ & $M_1$ & $E_6(1)$ \\ 
    $N_1$ & $1$  & $(x,y)$ & $M_1$ & - \\ 
    $A$ & $2$ & $(x,y)$ & $M_1$ & - \\ 
     $B$ & $1$ & $(x^2,y^2,xy)$ &$B$ & smooth \\ 
      $X$ & $2$ & $(x,y)$ & $M_1$ & - \\ 
    $M_2$ & $1$ & $(x,y^2)$ & $M_2$ & $A_{2}$ \\ 
    \hline
  \end{tabular}
  \vspace{0.1cm}
  \caption{Indecomposable $\CM$-modules over $R$ of type $E_6$.}
  \label{tab:E6}
\end{table} 

From the trace ideals we get the following list of modules living over each ring: \\
(1) $R$: all modules live over $R$. \\
(2) $M_1$: 
all indecomposables but $R$ are modules over $M_1$. \\
(3) $M_2$ (i.e., the $A_2$-singularity): $M_2$, $B$. \\
(4) $B$ (i.e., the normalization, a line): $B$. 
\end{ex}

\begin{ex} $E_8$: Let us first determine the rings dominating $R=k\{x,y\}/(x^3+y^5)$. Since $R$ is irreducible, one can determine the overrings from the semigroup of $R$ and then compare ranks with the indecomposable $\CM$-modules in Yoshino's list: the semigroup of $R$ is generated by $1, t^3,t^5$. Thus monomial rings between $R$ and its normalization are the ones with semigroup generated by $1, t^3,t^5, t^7$ (this is $E_8(1)$ in the notation of \cite{FruehbisKrueger}, the endomorphism ring of the maximal ideal), by $1, t^3, t^4, t^5$ (this is $E_6(1)$, which has already appeared in the previous example), by $1, t^2, t^3$ (this is the $A_2$-singularity) and finally by $1,t$, corresponding to the normalization. \\
{\bf Claim.} These are all possible rings dominating $R$: \\
{\it Proof of claim:} If $R'$ is dominating $R$, it has to have rank one over $R$. Moreover one has $\End_R R' \cong R'$. From Yoshino's list one sees that the rank one modules over $R$ are: $M_1$, $N_1$, $M_2$, $N_2$ and $A_2$. In the normalization chain we have seen that $\End_R \mf{m} \cong M_1$. Taking the endomorphism ring of the maximal ideal of $M_1$, the second ring is a cusp, isomorphic to the module $A_1$ and the third ring in the chain is the normalization, isomorphic to $A_2$. By computing endomorphism rings, we find that $\End_R N_1 \cong M_1$ and $\End_R N_2 \cong M_2$, so $N_1$ and $N_2$ are not isomorphic to overrings of $R$. In order to see that $M_2$ is isomorphic to the ring $R'=k\{t^3,t^4,t^5\}$, one can compute its minimal projective resolution and compare it to the projective resolution of the ideal $(x,y^3)$ in $R$, which is isomorphic to $R'$ (as an $R$-module). This exhausts all possibilities. \qed \\
So we obtain the following list of rings dominating $E_8$:
$$R \subseteq M_1 \subseteq M_2 \subseteq A_1 \subseteq A_2$$
Now from the trace ideals we get a list of the modules defined over the rings dominating $R$: \\
 (1) $R$: all modules live over $R$.\\
 (2) $M_1$ (i.e., $E_8(1)$, the singular space curve $k\{t^3,t^5,t^7\}$): all modules but $R$.  \\
 (3) $M_2$ (i.e., $E_6(1)$, the singular cubic $k\{t^3,t^4,t^5\}$): $M_2$, $N_2$, $C_2$, $Y_2$, $A_1$, $A_2$ (look at trace ideals!) \\
 (4) $A_1$ (i.e., the cusp $k\{t^2,t^3\}$): $A_1$, $A_2$. \\
 (5) $A_2$ (i.e., the normalization $k\{t\}$): $A_2$.
 \end{ex}
 
\subsection{Reducible ADE curves}
 
 \begin{ex} \label{Ex:Dnodd} $D_n$, $n$ odd: The $D_n$ singularity has coordinate ring $R=k\{x,y\}/(x^2y+y^{n-1})$. Geometrically, it is the union of the $A_{n-1}$-singularity $x^2+y^{n-2}=0$ with the line $y=0$. The module $B$ and all the $M_i$ are of rank $(0,1)$ and the rank of the smooth component $A$ is $(1,0)$. The ranks of $X_i$ and $Y_i$ are $(1,1)$. Geometrically, $M_i$ is isomorphic to an $A_{n-2i-3}$-singularity (including $M_{\frac{n-3}{2}}$, which is smooth) and the $X_i$ are $A_{n-2i-1} \vee L$ singularities, whereas the $Y_i$ are isomorphic to the respective canonical modules of the curves $X_i$. 
 Here is a table showing the ranks of the modules, their trace ideals, the largest overring dominating $R$ and (if they are isomorphic to a ring) the type of singularity:
 \begin{table}[h!]
  \centering
  \begin{tabular}{ |c|c|c|c|c|}
    \hline
     Module & rank & trace & largest ring & singularity type \\ 
   \hline \hline 
    $R$ & $(1,1)$ & $(1)$ & $R$ & $D_n$ \\ 
       $X_i$, $1 \leq i \leq \frac{n-1}{2}$ & $(1,1)$ & $(x,y^i)$ & $X_i$ & $A_{n-2i-1} \vee L$ \\ 
    $Y_i$, $1 \leq i \leq \frac{n-3}{2}$ & $(1,1)$ & $(x,y^i)$ & $X_i$ & - \\ 
    $M_i$, $1 \leq i \leq \frac{n-3}{2}$ & $(0,1)$ & $(xy,y^{i+1})$ & $A \oplus M_i$ & $A_{n-2i-3}$ \\ 
     $N_i$, $1 \leq i \leq \frac{n-3}{2}$ & $(2,1)$ & $(x,y^i)$ & $X_i$ & - \\ 
      $A$ & $(1,0)$ & $(x^2+y^{n-2})$ & $A \oplus M_{\frac{n-3}{2}}$ & smooth \\ 
    $B$ & $(0,1)$ & $(y)$ & $A \oplus B$ & $A_{n-3}$ \\ 
    \hline
  \end{tabular}
  \vspace{0.1cm}
  \caption{Indecomposable $\CM$-modules over $R$ of type $D_n$, $n$ odd.}
  \label{tab:Dnodd}
\end{table} 

Thus the overrings of $R$ are:
 $$R, A \oplus B, A \oplus M_i \text{ for } 1 \leq i \leq \frac{n-3}{2}, X_i \text{ for } 1 \leq i \leq \frac{n-1}{2}.$$
 By computing the trace ideals, one obtains the partial orders for the overrings:
\begin{equation*}
\begin{tikzpicture}[description/.style={fill=white,inner sep=2pt}]    \label{Diag:trace1}
    \matrix (m) [matrix of math nodes, row sep=3em,
    column sep=2.5em, text height=1.5ex, text depth=0.25ex]
    { R &   X_1 & X_2 & X_3 & \cdots & \phantom{a}X_{\frac{n-1}{2}} \\
       & A\oplus B  & A \oplus M_1 & A \oplus M_2 & \cdots & A \oplus M_{\frac{n-3}{2}} \\ };
            \path[right hook->] (m-1-1) edge (m-1-2);
             \path[right hook->] (m-1-2) edge (m-1-3);
             \path[right hook->] (m-1-3) edge (m-1-4);
             \path[right hook->] (m-1-4) edge (m-1-5);
             \path[right hook->] (m-1-5) edge (m-1-6);
           
             \path[right hook->] (m-1-2) edge (m-2-2);
             \path[right hook->] (m-1-3) edge (m-2-3);
             \path[right hook->] (m-1-4) edge (m-2-4);
             \path[right hook->] (m-1-6) edge (m-2-6);
             
             \path[right hook->] (m-2-2) edge (m-2-3);
             \path[right hook->] (m-2-3) edge (m-2-4);
             \path[right hook->] (m-2-4) edge (m-2-5);
             \path[right hook->] (m-2-5) edge (m-2-6);
      \end{tikzpicture}
             \end{equation*}

 The list of modules living over the rings dominating $R$ is as follows:\\
(1) $R$: all modules. \\
(2) $X_1$ (i.e., the $A_{n-3} \vee L$-space curve singularity): all modules but $R$. \\
(3) $X_i$, $2 \leq i \leq \frac{n-1}{2}$ (i.e., the $A_{n-2i-1} \vee L$-space curve singularity, including $X_{\frac{n-1}{2}}$, the $A_1$-singularity): $A$, $M_j$ for $i-1 \leq j \leq \frac{n-3}{2}$, $N_j$, $i \leq j \leq \frac{n-3}{2}$, $X_j$, $i \leq j \leq \frac{n-1}{2}$ and $Y_j$, $i \leq j \leq \frac{n-3}{2}$. (see inductively) \\
(4) $A \oplus B$ (i.e., disjoint union of a line and an $A_{n-3}$-singularity): $A$, $B$, $M_j$, $j \geq 1$. \\
(5) $A \oplus M_i$, $1 \leq i \leq \frac{n-3}{2}$ (i.e., disjoint union of a line and an $A_{n-2i-3}$-singularity): $A$, $M_j$, $i \leq j \leq \frac{n-3}{2}$. \\

The only modules not having full support on $R$ are $A,B$ and $M_i$, $1 \leq i \leq \frac{n-3}{2}$. The centres of endomorphism rings of modules only supported on one of the two components are as follows: $\End_R(A)=A$ is commutative, thus equal to its centre. Any module supported only on the singular component $x^2+y^{n-2}$ is of the form $B \oplus \bigoplus_{i \in I} M_i$ or $\bigoplus_{i \in I} M_i$, where $I \subseteq \{1, \ldots, \frac{n-3}{2}\}$. The centre is $B$ in the first case and $\min_i \{ M_i: i \in I \}$ in the second case.
 \end{ex}

 \begin{ex} \label{Ex:Dneven} 
 $D_n$, $n \geq 4$ even: The coordinate ring is $R=k\{x,y\}/(y(x^2+y^{n-2}))$. As in the previous examples, the rings dominating $R$ can be computed by rank considerations. We obtain the following list: 
 $$R, X_i \text{ for } 1 \leq i \leq \frac{n-2}{2}, A \oplus B, A\oplus M_i \text{ for } 1 \leq i \leq \frac{n-4}{2}, C_- \oplus D_-, C_+ \oplus D_+, A \oplus D_+ \oplus D_-.$$
 They fit together in the partial order: 
 
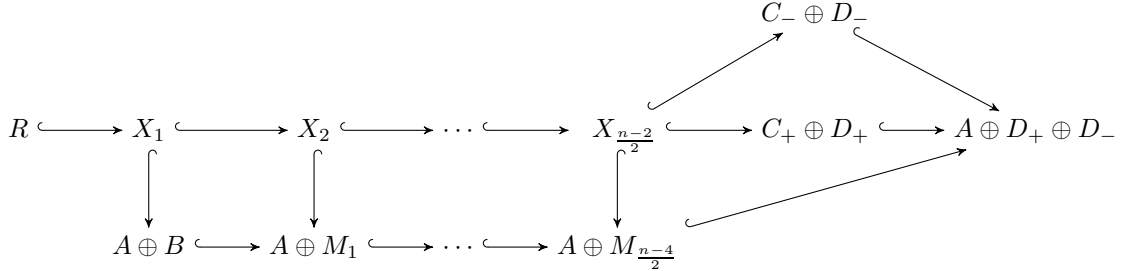
\begin{equation*}
\begin{tikzpicture}[description/.style={fill=white,inner sep=2pt}]    \label{Diag:trace1}
    \matrix (m) [matrix of math nodes, row sep=3em,
    column sep=2.5em, text height=1.5ex, text depth=0.25ex]
    {  & & & & & C_- \oplus D_- & \\
R &   X_1 & X_2 & \cdots  & \phantom{a}X_{\frac{n-2}{2}} & C_+ \oplus D_+   & A \oplus D_+ \oplus D_- \\
       & A\oplus B  & A \oplus M_1 & \cdots &  A \oplus M_{\frac{n-4}{2}}  & & \\ };
            \path[right hook->] (m-2-5) edge (m-1-6);
             \path[right hook->] (m-1-6) edge (m-2-7);
              \path[right hook->] (m-2-1) edge (m-2-2);
             
             \path[right hook->] (m-2-2) edge (m-2-3);
             \path[right hook->] (m-2-3) edge (m-2-4);
             \path[right hook->] (m-2-4) edge (m-2-5);
             \path[right hook->] (m-2-5) edge (m-2-6);
             \path[right hook->] (m-2-6) edge (m-2-7);
             
             \path[right hook->] (m-3-2) edge (m-3-3);
             \path[right hook->] (m-3-3) edge (m-3-4);
             \path[right hook->] (m-3-4) edge (m-3-5);
             
             \path[right hook->] (m-2-2) edge (m-3-2);
             \path[right hook->] (m-2-3) edge (m-3-3);
             \path[right hook->] (m-2-5) edge (m-3-5);
             
             \path[right hook->] (m-3-5) edge (m-2-7);
      \end{tikzpicture}
             \end{equation*}

 (If $n=4$, the module $M_{\frac{n-4}{2}}$ has to be replaced by $B$).
  Here is a table showing the ranks of the indecomposable $\CM(R)$-modules, their trace ideals, the largest ring dominating $R$ over which the respective module is still defined and (if they are isomorphic to a ring) the type of singularity:

  \begin{table}[h!]
  \centering
  \begin{tabular}{ |c|c|c|c|c|}
    \hline
     Module & rank & trace & largest ring & singularity type \\ 
   \hline \hline 
    $R$ & $(1,1,1)$ & $(1)$ & $R$ & $D_n$ \\ 
       $X_i$, $1 \leq i \leq \frac{n-2}{2}$ & $(1,1,1)$ & $(x,y^i)$ & $X_i$ & $A_{n-2i-1} \vee L$ \\ 
    $Y_i$, $1 \leq i \leq \frac{n-2}{2}$ & $(1,1,1)$ & $(x,y^i)$ & $X_i$ & - \\ 
    $M_i$, $1 \leq i \leq \frac{n-4}{2}$ & $(0,1,1)$ & $(xy,y^{i+1})$ & $A \oplus M_i$ & $A_{n-2i-3}$ \\ 
     $N_i$, $1 \leq i \leq \frac{n-4}{2}$ & $(2,1,1)$ & $(x,y^i)$ & $X_i$ & - \\ 
      $A$ & $(1,0,0)$ & $(x^2+y^{n-2})$ & $A \oplus D_- \oplus D_+$ & smooth \\ 
    $B$ & $(0,1,1)$ & $(y)$ & $A \oplus B$ & $A_{n-3}$ \\ 
      $C_+$ & $(1,0,1)$ & $(x-iy^{\frac{n-2}{2}})$ & $C_+ \oplus D_+$ & $A_{1}$ \\
        $C_-$ & $(1,1,0)$ & $(x+iy^{\frac{n-2}{2}})$ & $C_- \oplus D_-$ & $A_{1}$ \\
    $D_+$ & $(0,0,1)$ & $(y(x+iy^{\frac{n-2}{2}}))$ & $A \oplus D_- \oplus D_+$ & smooth \\ 
     $D_-$ & $(0,1,0)$ & $(y(x-iy^{\frac{n-2}{2}}))$ & $A \oplus D_- \oplus D_+$ & smooth \\ 
    \hline
  \end{tabular}
 \vspace{0.1cm}
  \caption{Indecomposable $\CM$-modules over $R$ of type $D_n$, $n$ even.}
\end{table}

 The list of modules living on the rings dominating $R$ is: \\
 (1) $R$: all modules. \\
 (2) $X_1$ (i.e., the space curve singularity $A_{n-3} \vee L$): all but $R$. \\
 (3) $X_i$, $2 \leq i \leq \frac{n-2}{2}$, (i.e., the space curve singularity $A_{n-2i-1} \vee L$): $X_j$ for $i \leq j \leq \frac{n-2}{2}$, $Y_j$ for $i \leq j \leq \frac{n-2}{2}$, $M_j$ for $i-1 \leq j \leq \frac{n-4}{2}$, $N_j$ for $i \leq j \leq \frac{n-4}{2}$, $A$, $C_+$, $D_+$, $C_-$, $D_-$. \\
 (4) $A \oplus B$ (i.e., the disjoint union of a line and an $A_{n-3}$-singularity): $A$, $B$, $M_j$ for $1 \leq j \leq \frac{n-4}{2}$, $D_+$, $D_-$. \\
 (5) $A \oplus M_i$ (i.e., the disjoint union of a line and an $A_{n-2i-3}$-singularity): $A$, $M_j$ for $i \leq j \leq \frac{n-4}{2}$, $D_+$, $D_-$. \\
 (6) $C_+ \oplus D_+$ (i.e., the disjoint union of a line and an $A_1$-singularity): $C_+$, $D_+$, $A$. \\
 (7) $C_- \oplus D_-$ (i.e., the disjoint union of a line and an $A_1$-singularity): $C_-$, $D_-$, $A$. \\
 (8) $A \oplus D_+ \oplus D_-$ (i.e., the disjoint union of three lines, the normalization): $A$, $D_+$, $D_-$.  \\

Again one can analyze the centres of the endomorphism rings which are are not of full support: these are the endomorphism rings of the modules: $A \oplus C_{+/-}$, $A \oplus D_{+/-}$, $M_i \oplus B$, $M_i \oplus D_{+/-}$, $B \oplus D_{+/-}$, $A \oplus C_{+} \oplus D_-$, $A \oplus C_- \oplus D_+$, $B \oplus D_- \oplus D_+$, $M_i \oplus D_+ \oplus D_-$. \\

Using Lemma \ref{Lem:support}, we can easily compute the centres of the other rings: $\End_R(A \oplus C_{+/-})=\End_{R/(y(x\pm iy^{\frac{n-2}{2}}))}(A \oplus C_{+/-})$. Since $C_{+ / -} \cong R/(y(x\pm iy^{\frac{n-2}{2}}))$ is an $A_1$ singularity, it follows that the global dimension of this ring is $3$ and its centre is $C_{+/-}$.  Moreover, $\End_R(A \oplus D_{+/-}) \cong A \oplus D_{+/-}$ is commutative 
and $\End_R(M_i \oplus B)=\End_B(M_i \oplus B)$ ($\gl$ of this ring is infinite!) has centre $B$, which is isomorphic to an $A_{n-3}$-singularity. Similarly $\End_R(M_i \oplus D_{+/-})=\End_{M_i}(M_i \oplus D_{+/-})$,  resp.~$\End_R(B \oplus D_{+/-})=\End_{B}(B \oplus D_{+/-})$ have  centre isomorphic to $M_i$  resp.~$B$ (and $\gl$ in all cases is $3$, since the modules are cluster tilting objects over the $A_{n-2i-3}$-singularities $M_i$  resp.~the $A_{n-3}$-singularity $B$). The two modules $A \oplus C_{+} \oplus D_{-}$, $A \oplus C_{-} \oplus D_{+}$ are supported on the $A_1$-singularities $C_+ \cong R/(y(x+ iy^{\frac{n-2}{2}}))$ and $C_- \cong R/(y(x- iy^{\frac{n-2}{2}}))$, so their centres are $C_+$ and $C_-$ and their global dimension is $2$ (representation generator for $A_1$-singularities). The module $B \oplus D_+ \oplus D_-$ is supported on the $A_{n-3}$-singularity $B$, thus by Lemma \ref{Lem:support} its centre is $B$ and the global dimension is infinite.
 \end{ex}

 \begin{ex} 
 $E_7$: Again we first determine the rings dominating $R=k\{x,y\}/(x^3+xy^3)$, i.e., the rings $R'$, satisfying $R \subseteq R' \subseteq \widetilde R$. 
 We compute again the ranks of the indecomposable from Yoshino's list, see Table \ref{tab:E7} for the list of rank $1$ modules. 
Here possible rings dominating $R$ are the $(1,1)$-modules and $A \oplus B$ and $A \oplus D$ and $M_2, N_2, Y_1, M_1, N_1$.
 A {\sc Singular} computation shows that the overrings $R_i$ appearing in Leuschke's normalization chain are: $R_1 \cong M_1$ (this is the singularity $E_7(1)$, again cf. Fr\"uhbis-Kr\"uger's table for the notation), $R_2 \cong Y_1$ (this is an $A_1$-singularity) and the normalization $R_3 \cong A \oplus D$. 
Again from computing endomorphism rings of the remaining candidates, it follows that $N_1$ and $N_2$ are not isomorphic to rings (they are isomorphic to the canonical modules of the non-Gorenstein rings $M_1$ and $M_2$). The only missing module is $M_2$: it is as ring isomorphic to the singularity $A_2 \vee L$, which  also appears as an overring of the $D_5$-singularity. As we have seen in Example \ref{Ex:Dnodd}, $A_2 \vee L$ is dominated by the $A_1$-singularity, the disjoint union of a line and the $A_2$-singularity and the normalization of $R$. But one can also see that $A_2 \vee L$ dominates $E_7(1)$ (from the description of the $\CM$-modules in \cite[p. 424, (2.3) (i)]{GreuelKnoerrer}): here $E_7(1)$ corresponds to $R + (t^4,0)R$ and $A_2 \vee L$ to $R + (t^2,0)R$ 
, which are contained in each other. Thus for $E_7$ the poset of overrings is: 
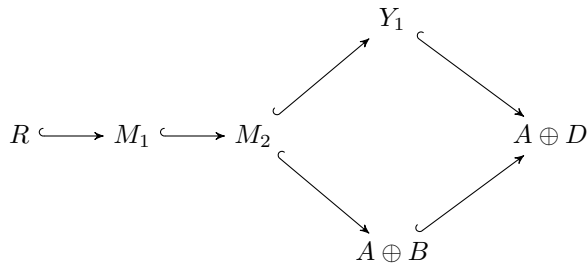
\begin{equation*}
\begin{tikzpicture}[description/.style={fill=white,inner sep=2pt}]    \label{Diag:trace1}
    \matrix (m) [matrix of math nodes, row sep=3em,
    column sep=2.5em, text height=1.5ex, text depth=0.25ex]
    { & & & Y_1 & \\
     R &   M_1 & M_2 &  & A \oplus D \\
       & & & A \oplus B  & \\ };
            \path[right hook->] (m-2-1) edge (m-2-2);
             \path[right hook->] (m-2-2) edge (m-2-3);
           
             \path[right hook->] (m-2-3) edge (m-1-4);
             \path[right hook->] (m-1-4) edge (m-2-5);
             \path[right hook->] (m-2-3) edge (m-3-4);
             \path[right hook->] (m-3-4) edge (m-2-5);
  \end{tikzpicture}
             \end{equation*}

Below is a table showing the ranks of the modules, their trace ideals, the largest overring and (if they are isomorphic to a ring) the type of singularity.
 \begin{table}[htbp]
  \centering
  \begin{tabular}{ |c|c|c|c|c|}
    \hline
     Module & rank & trace & largest ring & singularity type \\ 
   \hline \hline 
    $R$ & $(1,1)$ & $(1)$ & $R$ & $E_7$ \\ 
       $M_1$ & $(1,1)$ & $\mf{m}$ & $M_1$ & $E_7(1)$ \\ 
    $N_1$ & $(1,1)$ & $\mf{m}$ & $M_1$ & - \\ 
    $X_1$ & $(2,2)$ & $\mf{m}$ & $M_1$ & - \\ 
     $X_2$ & $(1,2)$ & $\mf{m}$ & $M_1$ & - \\ 
      $X_3$ & $(2,2)$ & $\mf{m}$ & $M_1$ & - \\ 
    $Y_3$ & $(2,2)$ & $\mf{m}$ & $M_1$ & - \\ 
       $C$ & $(2,1)$ & $\mf{m}$ & $M_1$ & - \\
        $M_2$ & $(1,1)$ & $(x,y^2)$ & $M_2$ & $A_2 \vee L$ \\ 
    $Y_2$ & $(2,1)$ & $(x,y^2)$ & $M_2$ & - \\ 
    $N_2$ & $(1,1)$ & $(x,y^2)$ & $M_2$ & - \\ 
           $Y_1$ & $(1,1)$ & $(x^2,xy,y^2)$ & $Y_1$ & $A_1$ \\ 
    $B$ & $(0,1)$ & $(x)$ & $A \oplus B$ & $A_2$ \\ 
    $D$ & $(0,1)$ & $(x^2,xy)$ & $A \oplus D$ & smooth \\ 
    $A$ & $(1,0)$ & $(x^2+y^3)$ & $A \oplus D$ & smooth \\ 
    \hline
  \end{tabular}
 \vspace{0.1cm}
  \caption{Indecomposable $\CM$-modules over $R$ of type $E_7$.}
 \label{tab:E7}
\end{table} 
 
 So we also can get the list of modules over each ring: \\
 (1) $R$: all modules live over $R$. \\
 (2) $M_1$ (i.e., the space curve singularity $E_7(1)$): all indecomposables except $R$ are modules over $M_1$. \\
 (3) $M_2$ (i.e., the space curve singularity $A_2 \vee L$): clearly $Y_1$ and $A$, $D$ are modules over $M_2$, $N_2$ corresponds to the canonical module of $M_2$. There are only two modules left, (this can be seen by comparing the AR-quivers of $A_2 \vee L$, which is just the stable quiver of the $D_5$-singularity, and the $E_7$ quiver). One has to be of rank $(0,1)$, which singles out $B$. The last one has to be rank $(2,1)$. There are two candidates: $C$ or $Y_2$. \\
From comparison of the  trace ideals, it follows that $Y_2$ is a module over $M_2$. (Here abuse of notation: $M_2$ is seen here as a ring!) \\
 Thus the modules living over $M_2$ are: $M_2$, $Y_1$, $A$, $D$, $N_2$, $B$ and $Y_2$. \\
 
 (4) $Y_1$ (i.e., an $A_1$-singularity): $Y_1$, $A$, $D$. \\
 (5) $A \oplus B$ (i.e., disjoint union of a line and an $A_2$-singularity): $A$, $B$, $D$. \\
 (6) $A \oplus D$ (i.e., the normalization, the disjoint union of two lines): $A$, $D$.  \\

Here one sees that the only modules which are not supported on the whole $\Spec R$, are $B, D, A$. For endomorphism rings involving only these modules, one can compute the centres individually: $Z(\End_R(B,B))=Z(B)=B$, $Z(\End_R(D,D)=D$, $Z(\End_R(A,A))=A$ [direct computation via exact sequences, e.g. $R \xrightarrow{ \cdot (x^2+y^3)} R \lra B \lra 0$. Apply $\Hom_R(-,B)$ and obtain $0 \lra \Hom_R(B,B) \lra B \xrightarrow{\cdot (x^2+y^3)} B$. Since multiplication by $(x^2+y^3)$ is $0$ on $B \cong R/(x^2+y^3)$, it follows that $\Hom_R(B,B) \cong B$.]  Since $B \oplus D$ is isomorphic to a representation generator for an $A_2$-singularity, it follows from Lemma \ref{Lem:support} that $\End_R(B \oplus D) \cong \End_B(B \oplus D)$ and so $\gl \End_R(D \oplus B)=2$ and $Z(\End_R(D \oplus B))=B$.
\end{ex}

\section{Acknowledgements} 

The authors want to thank Hailong Dao, Osamu Iyama and Michael Wemyss for very helpful discussions. Moreover, thanks to Ragnar Buchweitz for his support. E.F.~ wants to thank the Mittag-Leffler Institute for its warm hospitality and for providing excellent work conditions. 


\def\cprime{$'$}

\end{document}